\documentclass[twoside,12pt,reqno]{article} %{amsart}
\usepackage{mathrsfs,times}
\date{}

\usepackage[english]{babel}
\usepackage[ansinew]{inputenc}
\usepackage{amsmath,amsthm,amsfonts,amssymb,graphicx}
\usepackage{graphics}
\usepackage{algorithm,algorithmic}
\usepackage{graphicx}
\usepackage{epstopdf}
\usepackage{caption}
\usepackage{cases}
\usepackage{amssymb,amsthm,amsmath}
\usepackage[numbers,sort&compress]{natbib}

 \newtheorem{theorem}{Theorem}[section]
 
 \newtheorem{lemma}[theorem]{Lemma}
 
 \theoremstyle{definition}
 \newtheorem{definition}[theorem]{Definition}
 \theoremstyle{remark}
 \newtheorem{remark}[theorem]{Remark}
 
 \theoremstyle{eg}

 \theoremstyle{fact}
 
\numberwithin{equation}{section}

\title{\Large\bf On the convergences and applications of the inertial-like proximal point methods for null point problems}

\author{Yan Tang$^{1,2}$, Shiqing Zhang*$^{1}$\\
$^{1}$College of Mathematics,Sichuan University, Chengdu,China\\
$^{2}$College of Mathematics and Statistics, Chongqing Technology \\
and Business University,Chongqing, China\\
 tangyan@ctbu.edu.cn; zhangshiqing@scu.edu.cn}

\begin{document}
\maketitle

\begin{abstract}

Motivated and inspired by the discretization of the nonsmooth system of a nonlinear oscillator with damping,  we propose what we call the {\it  inertial-like proximal} point algorithms for finding the null point of the sum of two maximal operators, which has many applied backgrounds, such as, convex optimization and variational inequality problems, compressed sensing etc..
The  common feature of the presented algorithms is using  the new inertial-like proximal point method  which  does  not involve the computation for the norm of the difference between two adjacent iterates $x_n$ and $x_{n-1}$ in advance, and avoids complex   inertial parameters satisfying the traditional and difficult checking  conditions. Numerical experiments are presented to illustrate the performances of the algorithms.

\vskip.2in \noindent
{\em 2000 Mathematics Subject Classification:} 65K10; 65K05; 47H10; 47L25.\\

\noindent {\em Keywords:}   null point problems;  inertial-like proximal methods; maximal monotone operators.

\end{abstract}

%%%%%%%%%%%%%%%%%%%%%%%%%%%%%%%%%%%%%%%%%%%%%%%%%%%%%%%%%%%%%%%%%%%%%%%%%%%%%%%%%%%%

\section{Introduction}

Let $A:H\rightarrow 2^H$ be a set-valued operator in real Hilbert space $H$.
\vskip 1mm

(1)\,  The {\it graph} of  $A$ is given by
\begin{equation*}
 gph A= \{ (x,y)\in H\times H : y\in Ax\}.
\end{equation*}

(2)\, The operator $A$ is said to be {\it monotone} if
 \begin{equation*}
 \langle x-y,u-v\rangle\geq 0, u\in Ax,v\in Ay,\,\,\, \forall x,y\in H.
 \end{equation*}

 (3)\, The operator $A$ is said to be {\it maximal monotone} if $gph A$ is not properly contained in the graph of any other monotone operator.
\vskip 1mm
Let $A, B$ be two maximal monotone operators in $H$. We are concern with the well-studied {\it null point problem} that is formulated as follows:
\begin{eqnarray}\label{Eq:inclusion}
0\in (A+B)x^*
\end{eqnarray}
with solution set denoted by $\Omega$.
\vskip 1mm
A special interesting case of (\ref{Eq:inclusion}) is a minimization of sum of two proper, lower semi-continuous and convex functions $f,g:H\rightarrow \mathbb{R}$, that is,
\begin{eqnarray}\label{P:minimizing}
\min_{x\in H} \{f(x)+g(x)\}.
\end{eqnarray}
So this is equivalent to (\ref{Eq:inclusion}) with $A=\partial f$ and $B=\partial g$ being the subdifferentials of $f$ and $g$, respectively.
\vskip 1mm
We recall the resolvent operator $J_r^{A}=(I+r A)^{-1},r>0,$ which is called the {\it backward operator} and plays an significant role in the approximation theory for zero points of maximal monotone operators. Due to the work of Aoyama et al. \cite{AKT2009}, we have the following properties:
\begin{align}\label{(2.3)}
\langle J_r^{A}x-y,x-J_r^{A}x\rangle\geq 0, y\in A^{-1}(0),
\end{align}
where $A^{-1}(0)=\{z\in H:0\in Az\}$. Moreover, the following key facts  that will be needed in the sequel.

{\bf Fact 1:}  The resolvent is not only always single-valued, but also firmly monotone:
\begin{align}\label{(2.3)}
\|x-y\|^2-\|(I-J_r^{A})x-(I-J_r^{A})y\|^2\geq \| J_r^{A}x- J_r^{A}y\|^2.
\end{align}

{\bf Fact 2:} Using the resolvent operator, we can write down {\it inclusion problem} (\ref{Eq:inclusion}) as a fixed point problem.
It is known that
\begin{eqnarray*}%%(1.5)
x^*=J_\lambda^{A}(I-\lambda{B})x^*, \lambda>0.
\end{eqnarray*}

Douglas-Rachford`s splitting method (DRSM) (or {\it forward-backward splitting method}) was first introduced in  \cite{DR1956} as an operator splitting technique to solve partial differential equations arising in heat conduction and soon later has been extended to find solutions for the sum of two maximal monotone operators  by Lions and Mercier \cite{LM1979}. Douglas-Rachford`s splitting method(DRSM) is formulated as 
\begin{equation}\label{Eq:1.7}
 x_{n+1}=J^A_{\lambda}(I-\lambda B)x_n,
\end{equation}
where $\lambda>0$, and $(I-\lambda B)$ is called the {\it forward operator}.
\vskip 1mm
 Basing on the  method (\ref{Eq:1.7}), many researchers improved and  modified the algorithms for the {\it inclusion problem} (\ref{Eq:inclusion}) and obtained nice results,
see e.g., Boikanyo \cite{B2016},  Dadashi and Postolache\cite{DP2020}, Kazmi and Rizvi \cite{KR2014}, Moudafi\cite{MT1997}\cite{M2018}, Sitthithakerngkiet et al. \cite{SDMK2018}.
\vskip 1mm
On the other hand, one classical way of looking at the null point problem $0\in Ax$ is to consider the  forward discretization for $\frac{dx}{dt}\approx\frac{x_{n+1}-x_{n}}{h_n}, \forall h_n>0$ in the {\it evolution system}:
\begin{eqnarray}\label{Eq:evol}
\begin{cases}
	\frac{dx}{dt}+Ax(t)\ni0, \\
	x(0)=x_0,
	\end{cases}
\end{eqnarray}
and then the {\it evolution system} is discretized as
\begin{equation*}
\frac{x_{n+1}-x_{n}}{h_n}+Ax_{n+1}\ni0\,\, \Longleftrightarrow\,\, (I+h_nF)x_{n+1}=x_{n},
\end{equation*}
which inspired Alvarez and Attouch \cite{AA2001} to introduce the  inertial method for the following nonsmooth case of a nonlinear oscillator with damping
\begin{eqnarray}\label{(1.4)}
\frac{d^2x}{dt^2}+\gamma\frac{dx}{dt}+A(x(t))=0, a.e. t\geq 0.
\end{eqnarray}
By discretizing, Alvarez and Attouch \cite{AA2001} obtained the implicit iterative sequence
\begin{eqnarray}\label{(1.5)}
x_{n+1}-x_n-\alpha_n(x_n-x_{n-1})+\lambda_n A(x_{n+1})\ni 0,
\end{eqnarray}
where $\alpha_n=1-\gamma h_n$ and $\lambda_n=h_n^2$, which yielded the {\it Inertial-Prox} algorithm
\begin{eqnarray}\label{(Eq:1.6)}
x_{n+1}=J_{\lambda_n}^A(x_n+\alpha_n(x_n-x_{n-1})),
\end{eqnarray}
the extrapolation term $\alpha_n(x_n-x_{n-1})$ is called {\it inertial term}.
\vskip 1mm
Note that when $\alpha_n\equiv 0$,  the recursion (\ref{(1.5)}) corresponds to the standard proximal iteration
\begin{eqnarray*}
x_{n+1}-x_n+\lambda_n A(x_{n+1})\ni 0,
\end{eqnarray*}
which has been well studied by Martinet \cite{M1970}and Moreau\cite{M1965} and other researchers, and
the weak convergence of $x_n$ to a solution of $0\in Ax$ has being well known since the classical work of  Rockafellar\cite{R1976}. 
\vskip 1mm
For ensuring the convergence of the {\it Inertial-Prox} sequence, Alvarez and Attouch\cite{AA2001} pointed the following key assumption: there exists $\alpha\in(0,1)$ such that $\alpha_n\in[0,\alpha]$ and
\begin{eqnarray}\label{(1.6)}
\sum_{n=1}^\infty \alpha_n\|x_n-x_{n-1}\|^2<\infty.
\end{eqnarray}

Inertial method is shown to have nice convergence properties in the field of continuous optimization and is studied intensively in split inverse problem by many authors soon later (because which could be utilized in some situations to accelerate the convergence of the sequences).   For some recent works applied to various fields, see Alvarez \cite{A2004}, Attouch et al. \cite{APR2014, ACPR2018, AC2016}, Ochs et al. \cite{OCBP2014,OBP2015}.
\vskip 1mm
Although Alvarez and Attouch \cite{AA2001} pointed that one can choose appropriate rule to enable the the assumption (\ref{(1.6)}) applicable, the parameter $\alpha_n$ involving the iterates $x_n$ and $x_{n-1}$ should be computed in advance. Namely, to make the condition (\ref{(1.6)}) hold, the researchers have to set constraints on the inertia coefficient $\alpha_n$ and estimate the value of the $\|x_n-x_{n-1}\|$ before choosing $\alpha_n$. In recent works, Gibali et al. \cite{GMN2018} improved the inertial control condition by constraining inertial coefficient $\alpha_n$ such that $0<\alpha_n<\bar{\alpha}_n$, where $\alpha\in(0,1),$ $\epsilon_n\in [0,\infty)$ and $\sum_{n=0}^\infty\epsilon_n<\infty$,
  \begin{eqnarray*}
\bar{\alpha}_n=
\begin{cases}
\min\{\alpha,\frac{\epsilon_n}{\|x_n-x_{n-1}\|^2}\}, & \hbox{if}\,\, x_n \neq x_{n-1},\\
\alpha , & \text{otherwise}.
\end{cases}
\end{eqnarray*}
More works on inertial methods, one can refer  Dang et al. \cite{DSX2017}, Moudafi et al.\cite{MT2013}, Suantai et al. \cite{SPC2017}, Tang \cite{Tang2019}, and therein.
\vskip 1mm
Theoretically, the condition (\ref{(1.6)}) on the parameter $\alpha_n$  is too strict for the convergence of the inertial algorithm. Practically, estimating the  value of the  $\|x_n-x_{n-1}\|$ before choosing the inertial parameter $\alpha_n$ may need large amount of computation.  The two drawbacks may make the {\it Inertial-Prox} method inconvenient in the practical test in the sense. So it is natural to think about the following question:
\vskip 1mm
{\bf Question 1.1} Can we delete the condition (\ref{(1.6)}) in inertial method?  Namely, can we construct a new inertial algorithm for solving (\ref{Eq:inclusion}) without any constraint on the  the inertial parameter  or the computation of norm of the difference between $x_n$ and $x_{n-1}$ before choosing the inertial parameter?
\vskip 1mm
The purpose of this paper is to present an affirmative answer to the above question. In this paper, we  study the convergence problem of a new inertial-like technique for the solution of the {\it null point problem} (\ref{Eq:inclusion}) without the assumption (\ref{(1.6)}) and
the prior computation of the  $\|x_n-x_{n-1}\|$ before choosing the inertial parameter $\theta_n$.
\vskip 1mm
The outline of the paper is as follows. In section 2, we collect definitions and results which
are needed for our further analysis.
In section 3, our novel approach for the {\it null point problem} is  introduced and analyzed,
the   convergence theorems  of the presented algorithms
are obtained.  Moreover, convex optimization and variational inequality problem are studied as the  applications of the {\it null point problem} in section 4.
Finally, in section 5,some numerical experiments, using the {\it inertial-like method}, are carried out in order to support our approach.

\vskip 2mm

  \section{Preliminaries}

Let $\langle \cdot,\cdot\rangle$ and $\|\cdot\|$ be the inner product and the induced norm in a Hilbert space $H$, respectively. For a sequence $\{x_n\}$ in $H$, denote $x_n\rightarrow x$ and $x_n \rightharpoonup x$ by
the strong and weak convergence to $x$ of $\{x_n\}$, respectively.
Moreover, the symbol $\omega_w(x_n)$ represents the $\omega$-weak limit set of $\{x_n\}$, that is,
\begin{eqnarray*}
\omega_w(x_n):=\{x\in H: x_{n_j}\rightharpoonup x \hspace{0.2cm} \mathrm{for}\hspace{0.2cm} \mathrm{some}\hspace{0.2cm} \mathrm{subsequence}\hspace{0.2cm} \{x_{n_j}\} \hspace{0.2cm}\mathrm{of} \hspace{0.2cm} \{x_n\}\}.
\end{eqnarray*}
The identity below is useful:
\begin{align}\label{(2.1)}
\nonumber
\|\alpha x+\beta y+\gamma z\|^2&=\alpha\|x\|^2+\beta\|y\|^2+\gamma\|z\|^2  \\
&\quad -\alpha\beta\|x-y\|^2-\beta\gamma\|y-z\|^2-\gamma\alpha\|x-z\|^2
\end{align}
for all $x,y,z \in R$ and  $\alpha+\beta+\gamma=1$.

%Moreover, the following inequalities hold:
% \begin{eqnarray*}
%\|x-y\|^2 = \|x\|^2-2\langle x,y\rangle+\|y\|^2, \forall x,y \in H,
%\end{eqnarray*}
% \begin{eqnarray}
%\|x+y\|^2\leq \|x\|^2+2\langle y,x+y\rangle, \forall x,y \in H,
%\end{eqnarray}

%{Definition 2.2}

\vskip 2mm

\noindent
\begin{definition}\label{Def:2.1}\,
 Let $H$ be a  real Hilbert space, $D\subset H$ and $T:D\rightarrow H$ some given operator.
\vskip 1mm

 (1)\, The operator $T$ is said to be {\it Lipschitz continuous} with constant $\kappa>0$ on $D$ if
\begin{equation*}
\|T(x)-T(y)\|\leq \kappa\|x-y\|,\,\,\, \forall x,y\in D.
\end{equation*}

% (2)\, The operator $T$ is said to be {\it hemicontinuous} if it is continuous along each line segment in $D$.

 (2)\, The operator $T$ is said to be {\it $\gamma-$cocoercive} if there exists $\gamma>0$ such that
 \begin{equation*}
\langle Tx-Ty, x-y\rangle\geq \gamma\|T(x)-T(y)\|^2,\,\,\,\forall x,y \in D.
\end{equation*}

%(4)\, The operator $T$ is said to be {\it averaged} if there exist a nonexpansive operator $h:D\rightarrow H$ and a number $c\in(0,1)$ such that
%\begin{equation}
%T=(1-c)I+ch.
%\end{equation}

%(5)\, The operator $T$ is said to be {\it firmly nonexpansive} if
%\begin{equation}
 %\langle Tx-Ty,x-y \rangle\geq \|T x- T y \|^2,\,\,\,\forall x,y \in D.
% \end{equation}
%\vskip 2mm
\end{definition}
%\noindent
\begin{remark}\label{re:2.1}\,
%{\bf Remark  2.1.} 
%(1)\, Every firmly nonexpansive mapping is nonexpansive, but the converse is not true.

%(2)\, The operator $T$ is firmly nonexpansive if and only if $I-T$ is firmly nonexpansive (see e.g., Censor et al.\cite[Lemma 2.3]{CGR2012}).

%(3)\, If $T_1$ and $T_2$ are averaged, then their composition $S=T_1\circ T_2$ is averaged (see e.g.,   Censor et al. \cite[Lemma 2.2]{CGR2012}).

(1)\, If $T$ is $\gamma-$cocoercive, then it is $\frac{1}{\gamma}-$Lipschitz continuous.

(2)\, From {\bf Fact 1}, we can conclude that $J_r^A$ is a non-expansive operator if $A$ is a maximal monotone mapping.
\end{remark}

\begin{definition}
Let $C$ be a nonempty closed convex subset of $H$. We use $P_C$ to denote the projection from $H$ onto $C$;
namely,
\begin{eqnarray*}
P_Cx=\arg\min\{\|x-y\|:y\in C\},\quad x\in H.
\end{eqnarray*}
\end{definition}
The following  significant characterization of the projection $P_C$ should be recalled : Given $x\in H$ and $y\in C$,
\begin{eqnarray}\label{(2.6)}
P_Cx=z\quad \Longleftrightarrow\quad \langle x-z,y-z\rangle\leq 0,\quad y\in C.
\end{eqnarray}
%{\bf Lemma 2.4}
\begin{lemma}\label{Lemma2.4} (Xu \cite{Xu2002})
Assume that $\{a_n\}$ is a sequence of nonnegative real numbers such that
 \begin{eqnarray*}
a_{n+1}\leq(1-\beta_n)a_n+\beta_n b_n+c_n,\quad n\geq 0,
\end{eqnarray*}
where $\{\beta_n\}$ is a sequence in $(0,1)$ and $\{c_n\}\subset(0,\infty)$ and $\{b_n\}\subset\mathbb{R}$ such that
\begin{itemize}
\item[(1)] $\sum_{n=1}^\infty\beta_n=\infty$;

\item[(2)] $\limsup_{n\rightarrow\infty}b_n\leq 0$ or $\sum_{n=1}^\infty\beta_n|b_n|<\infty$;
\item[(3)] $\sum_{n=1}^\infty c_n<\infty$.
\end{itemize}
Then $\lim_{n\rightarrow\infty}a_n=0$.
\end{lemma}

%{\bf Lemma 2.5}
\begin{lemma}\label{Lemma2.5} (see e.g., Opial \cite{O1967})
Let $H$ be a real Hilbert space and $\{x_n\}$ be a bounded sequence in $H$.
Assume there exists a nonempty subset $S \subset H$ satisfying the properties:
\begin{itemize}
\item[(i)] $\lim_{n\rightarrow\infty}\|x_n-z\|$ exists for every $z\in S$,

\item[(ii)] $\omega_w(x_n)\subset S$.
\end{itemize}
 Then, there exists $\bar{x}\in S$ such that $\{x_n\}$ converges weakly to $\bar{x}$.
\end{lemma}

\begin{lemma}\label{Lemma2.6}
(Maing\'e \cite{M2008}) Let $\{\Gamma_n\}$ be a sequence of real numbers that does not decrease at the infinity
in the sense that there exists a subsequence $\{\Gamma_{n_j}\}$ of $\{\Gamma_n\}$ such that
$\Gamma_{n_j}<\Gamma_{n_j+1}$ for all $j\geq 0$. Also consider the sequence
 of integers $\{\sigma(n)\}_{n\geq n_0}$ defined by
\begin{eqnarray*}
\sigma(n)=\max\{k\leq n:\Gamma_k\leq\Gamma_{k+1}\}.
\end{eqnarray*}
Then $\{\sigma(n)\}_{n\geq n_0}$ is a nondecreasing sequence verifying $\lim_{n\rightarrow\infty}\sigma(n)=\infty$ and,
for all $n\geq n_0$,
\begin{eqnarray*}
\max\{\Gamma_{\sigma(n)},\Gamma_n\}\leq \Gamma_{\sigma(n)+1}.
\end{eqnarray*}
\end{lemma}

\section{Main Results}
\subsection{Motivation of Inertial-Like Proximal Technique}
Inspired and motivated by the  discretization (\ref {(1.5)}), we consider the following  iterative sequence
 \begin{align*}
x_{n+1}-x_{n-1}-\theta_n(x_n-x_{n-1})+\lambda_n A(x_{n+1})\ni 0,
\end{align*}
where $x_0$, $x_1$ are two arbitrary initial points, and $\{\lambda_n\}$ is a real nonnegative number sequence. This recursion can be rewritten as
\begin{align}\label{(2.4)}
\begin{cases}
w_n=x_{n-1}+\theta_n(x_n-x_{n-1})\\
x_{n+1}=J_{\lambda_n}^{A}w_n.
\end{cases}
\end{align}

The discretization  sequence $\{x_n\}$ in (\ref{(2.4)}) always exists because the sequence $\{x_n\}$ satisfying (\ref{(1.5)}) always exists for any choice of the sequence $\{\alpha_n\}$ according to Alvarez and Attouch \cite{AA2001}. In addition, it can be deduced that if $\theta_n=1+\alpha_n$,  the formula (\ref{(2.4)}) can cover Alvarez and Attouch`s {\it Inertial-Prox} algorithm.  More relevantly,  the inertial coefficient $\theta_n$ can be considered 1 in our new inertial proximal point algorithms.
We thus obtain what we call the {\it Inertial-Like Proximal} point algorithm.

\subsection{Some Conventions and  Inertial-like Proximal  Algorithms}

{\bf C1:}Throughout the rest
of this paper, we always assume that $H$ is a Hilbert space.
We rephrase the {\it null point problem} as follows:
 \begin{eqnarray}\label{(3.2)}
0\in (A+B)x^*
\end{eqnarray}
where  $A, B:H\rightarrow 2^H$ are  two maximal monotone set-valued operators with $B$ $\gamma-$cocoercive. \\
{\bf C2:} Denote by $\Omega$ the solution set of the {\it null point problem}; namely,
\begin{eqnarray*}
\Omega=\left\{x^*\in H:  0\in (A+B)x^* \right\}
\end{eqnarray*}
and we always assume $\Omega\not=\emptyset$.

Now, combining the Fact 2 and the inertial-like technique (\ref{(2.4)}), we introduce the following algorithms.\\
$\line(1,0){450}$\\
{\bf Algorithm 3.1}\\
$\line(1,0){450}$\\

  {\bf Initialization:}  Choose a positive sequence $\{\theta_n\}\subset[0,1]$. Select arbitrary initial points $x_0$, $x_1$.

 {\bf Iterative Step:}  After the $n$-iterate $x_n$ is constructed, compute
 \begin{eqnarray}\label{(3.3)}
 w_n=x_{n-1}+\theta_n(x_n-x_{n-1}),\hspace{0.2cm}n\geq 1,
 \end{eqnarray}
and define the $(n+1)$th iterate by
\begin{eqnarray}\label{(3.4)}
x_{n+1}=J_{\tau_n}^{A}(I-\tau_n B)w_n.
\end{eqnarray}
$\line(1,0){450}$

%\begin{remark}\label
{\bf Remark 3.1}
It is not hard to find that if  $w_n=x_{n+1}$, for some $n\ge 0$, then $x_{n+1}$
is a solution of the {\it inclusion problem} (\ref{(3.2)}), and the iteration process is terminated in finite iterations. If  $\theta_n=1$, Algorithm 3.1 reduces to the general {\it forward-backward} algorithm in Moudafi\cite{M2018}.\\
%\end{remark}
$\line(1,0){450}$\\
{\bf Algorithm 3.2}\\
$\line(1,0){450}$\\

{\bf Initialization:}  Choose a  sequence $\{\theta_n\}\subset[0,1]$ satisfying one of the three cases: {(\bf I.)} $\theta_n \in (0,1)$ such that $\underline{\lim}_{n\rightarrow \infty}\theta_n(1-\theta_n)>0$; ({\bf II.}) $\theta_n\equiv 0$; ({\bf III.}) $\theta_n\equiv1$. Choose $\{\alpha_n\}$ and $\{\beta_n\}$  in $(0,1)$ such that
\begin{eqnarray*}
\underline{\lim}_{n\rightarrow\infty}\alpha_n>0;\overline{\lim}_{n\rightarrow\infty}\alpha_n<1;\hspace{0.2cm}\lim_{n\rightarrow\infty}\beta_n=0,\quad \sum_{n=0}^\infty\beta_n=\infty.
\end{eqnarray*}
Select arbitrary initial points $x_0$, $x_1$.

{\bf Iterative Step:}  After the $n$-iterate $x_n$ is constructed, compute
 \begin{eqnarray*}
 w_n=x_{n-1}+\theta_n(x_n-x_{n-1}),\hspace{0.2cm}
 \end{eqnarray*}
and define the $(n+1)$th iterate by
     \begin{eqnarray}\label{(3.5)}%%(3.6)
x_{n+1}=(1-\alpha_n-\beta_n)w_n+\alpha_n J_{\tau_n}^{A}(I-\tau_n B)w_n.
\end{eqnarray}
$\line(1,0){450}$

{\bf Remark 3.2.}
%\begin{remark}\label{Remark3.4}
In the subsequent convergence analysis, we will always assume that the two algorithms generate an infinite
sequence, namely, the algorithms are not terminated in finite iterations.
In addition, in the simulation experiments, in order to be practical, we will give a stop criterion to end the iteration for practice. Otherwise, set $n:=n+1$ and return to Iterative Step.

%\end{remark}
%{\bf Stop Criterion:} If  $\theta^2(x_n)=0$ or $\theta^2(y_n)=0$, then the iterative process stop.
%Otherwise, set $n:=n+1$ and return to Iterative Step.

%In the above algorithms, $A_j^*$ denotes the adjoint of $A_j$. Especially, in finite dimensional spaces,
%$A_j^*=A_j^T$, the transpose of $A_j$.
%The parameter $\delta$ will be defined in the following convergence analysis.

\subsection{ Convergence Analysis of Algorithms}

%{\bf Theorem 3.5}
\begin{theorem}\label{Theorem3.5}
If the assumptions ${\bf C1}-{\bf C2}$ are satisfied and $\tau_n\in(\epsilon,2\gamma-\epsilon)$ for some given $\epsilon>0$ small enough, we obtain the weak convergence result, namely the sequence $\{x_n\}$  generated by Algorithm 3.1 converges weakly to a point $\bar{x} \in \Omega$.
\end{theorem}

\begin{proof}
Without loss of generality, we take $z\in \Omega$ and then we have get $z=J_{\tau_n}^A(I-\tau_n B)z$ from {\bf Fact 2}.
\vskip 2mm
It turns out from  (\ref{(2.1)}) and (\ref{(3.3)})  that
\begin{eqnarray}\label{(3.6)}
\|w_n-z\|^2&=&\|x_{n-1}+\theta_n(x_n-x_{n-1})-z\|^2 \nonumber\\
&=&\theta_n\|x_n-z\|^2+(1-\theta_n)\|x_{n-1}-z\|^2-\theta_n(1-\theta_n)\|x_n-x_{n-1}\|^2.
\end{eqnarray}
Since $J_r^A$ is firmly  nonexpansive, it follows from (\ref{(3.4)}) and Fact 1 that
\begin{eqnarray}\label{(3.7)}
\|x_{n+1}-z\|^2&=&\|J_{\tau_n}^{A}(I-\tau_n B)w_n-z\|^2 \nonumber\\
&\leq&\|(I-\tau_n B)w_n-(I-\tau_n B)z\|^2-\|(I-J_{\tau_n}^{A})(I-\tau_n B)w_n-(I-J_{\tau_n}^{A})(I-\tau_n B)z\|^2\nonumber \\
&=&\|(w_n-z)-\tau_n (Bw_n- Bz)\|^2-\|w_n-x_{n+1}-\tau_n(Bw_n-Bz)\|^2 \nonumber\\
&=&\|w_n-z\|^2-2\langle w_n-z,\tau_n (Bw_n- Bz)\rangle+ \tau_n^2\|Bw_n- Bz\|^2\nonumber\\
&&-\|w_n-x_{n+1}\|^2+2\langle w_n-x_{n+1},\tau_n (Bw_n- Bz)\rangle-\tau_n^2\|Bw_n- Bz\|^2\nonumber\\
&=&\|w_n-z\|^2-\|w_n-x_{n+1}\|^2-2\langle w_n-z,\tau_n (Bw_n- Bz)\rangle\nonumber\\
&&+2 \tau_n\langle w_n-x_{n+1},Bw_n- Bz\rangle.
\end{eqnarray}
It follows from the fact  $B$ is $\gamma-$cocoercive that
\begin{eqnarray*}
\langle Bw_n-Bz,w_n-z\rangle\geq \gamma\|Bw_n-Bz\|^2,
\end{eqnarray*}
 so we have from (\ref{(3.7)}) that
 \begin{eqnarray*}
\|x_{n+1}-z\|^2&\leq&\|w_n-z\|^2-\|w_n-x_{n+1}\|^2-2\gamma\tau_n\|Bw_n-Bz\|^2\\
&&+2 \tau_n\langle w_n-x_{n+1},Bw_n- Bz\rangle.
\end{eqnarray*}

 On the other hand, we have
 \begin{eqnarray*}
&&2\gamma\tau_n\Big\|Bw_n-Bz-\frac{w_n-x_{n+1}}{2\gamma}\Big\|^2\\
&=&2\gamma\tau_n\Big(\|Bw_n-Bz\|^2+\Big\|\frac{w_n-x_{n+1}}{2\gamma}\Big\|^2-2\Big\langle Bw_n-Bz,\frac{w_n-x_{n+1}}{2\gamma}\Big\rangle\Big)\\
&=&2\gamma\tau_n\|Bw_n-Bz\|^2+\frac{\tau_n}{2\gamma}\|w_n-x_{n+1}\|^2-2\tau_n\langle Bw_n-Bz,w_n-x_{n+1}\rangle
\end{eqnarray*}
and, furthermore,
  \begin{eqnarray*}
&&2\gamma\tau_n\|Bw_n-Bz\|^2-2\tau_n\langle Bw_n-Bz,w_n-x_{n+1}\rangle\\
&=&
2\gamma\tau_n\Big\|Bw_n-Bz-\frac{w_n-x_{n+1}}{2\gamma}\Big\|^2-\frac{\tau_n}{2\gamma}\|w_n-x_{n+1}\|^2.
\end{eqnarray*}
\vskip 2mm

 Hence we obtain from  (\ref{(3.6)}), (\ref{(3.7)})  that
   \begin{eqnarray}\label{(3.8)}
\nonumber\|x_{n+1}-z\|^2&\leq&\|w_n-z\|^2-\|w_n-x_{n+1}\|^2-2\gamma\tau_n\|Bw_n-Bz\|^2\\ \nonumber
&&+2 \tau_n\langle w_n-x_{n+1},Bw_n- Bz\rangle \\\nonumber
&=&\|w_n-z\|^2-2\gamma\tau_n\Big\|Bw_n-Bz-\frac{w_n-x_{n+1}}{2\gamma}\Big\|^2\\ \nonumber
&&+(\frac{\tau_n}{2\gamma}-1)\|w_n-x_{n+1}\|^2\\ \nonumber
&=&\theta_n\|x_n-z\|^2+(1-\theta_n)\|x_{n-1}-z\|^2-\theta_n(1-\theta_n)\|x_n-x_{n-1}\|^2\\
&&-2\gamma\tau_n\Big\|Bw_n-Bz-\frac{w_n-x_{n+1}}{2\gamma}\Big\|^2+(\frac{\tau_n}{2\gamma}-1)\|w_n-x_{n+1}\|^2.
\end{eqnarray}
Since $\tau_n\leq (\epsilon,2\gamma-\epsilon)$, it follows from (\ref{(3.8)}) that
 \begin{eqnarray*}
\|x_{n+1}-z\|^2&\leq&\theta_n\|x_n-z\|^2+(1-\theta_n)\|x_{n-1}-z\|^2\\
&\leq&\max\{\|x_n-z\|^2,\|x_{n-1}-z\|^2\}\\
&\leq&\cdots\\
&\leq&\max\{\|x_1-z\|^2,\|x_0-z\|^2\},
\end{eqnarray*}
which means that the sequence $\{\|x_n-z\|\}$ is bounded, and  so in turn $\{w_n\}$ is.
\vskip 2mm
Now we claim that  there exists the limit of the sequence $\{\|x_n-z\|\}$. For this purpose,  two situations are discussed as follows:.
\vskip 2mm
{\bf Case 1.} There exists an integer $N_0\geq 0$ such that  $\|x_{n+1}-z\|\leq\|x_n-z\|$ for all $n\geq N_0$. Then  there exists the limit of the sequence $\{\|x_n-z\|\}$, denoted by $l=\lim_{n\rightarrow}\|x_n-z\|^2$, and so
 $$\lim_{n\rightarrow\infty}(\|x_n-z\|^2-\|x_{n+1}-z\|^2)=0.$$

  In addition, we have
  \begin{eqnarray*}
\sum_{n=0}^\infty(\|x_{n+1}-z\|^2-\|x_n-z\|^2)=\lim_{n\rightarrow \infty}(\|x_{n+1}-z\|-\|x_0-z\|)<\infty
\end{eqnarray*}
and therefore from (\ref{(3.8)}) we get
\begin{eqnarray*}
\lim_{n\rightarrow}\theta_n(1-\theta_n)\|x_n-x_{n-1}\|^2\leq (\|x_n-z\|^2-\|x_{n+1}-z\|^2)+(1-\theta_n)(\|x_{n-1}-z\|^2-\|x_n-z\|^2)
\end{eqnarray*}
and so
\begin{eqnarray*}
\lim_{n\rightarrow}\theta_n(1-\theta_n)\|x_n-x_{n-1}\|^2=0,\quad \sum_{n=0}^\infty\theta_n(1-\theta_n)\|x_n-x_{n-1}\|^2<\infty.
\end{eqnarray*}

Now it remains to show that
\begin{equation*}
\omega_w(x_n)\subset \Omega.
\end{equation*}
Since the sequence $\{x_n\}$ is bounded, let $\bar{x}\in \omega_w(x_n)$ and $\{x_{n_k}\}$
be a subsequence of $\{x_n\}$ weakly converging to $\bar x$.  To this end, it remains to verify that $\bar{x}\in (A+B)^{-1}(0)$.
\vskip 2mm
 Notice  again (\ref{(3.8)}),  we have
  \begin{eqnarray*}
&&2\gamma\tau_n\Big\|Bw_n-Bz-\frac{w_n-x_{n+1}}{2\gamma}\Big\|^2+(1-\frac{\tau_n}{2\gamma})\|w_n-x_{n+1}\|^2+
\theta_n(1-\theta_n)\|x_n-x_{n-1}\|^2\\
&\leq&\theta_n\|x_n-z\|^2+(1-\theta_n)\|x_{n-1}-z\|^2-\|x_{n+1}-z\|^2\\
&=&(\|x_n-z\|^2-\|x_{n+1}-z\|^2)+(1-\theta_n)(\|x_{n-1}-z\|^2-\|x_n-z\|^2),
\end{eqnarray*}
 which means that
 $$\Big\|Bw_n-Bz-\frac{w_n-x_{n+1}}{2\gamma}\Big\|^2\rightarrow 0; $$
 and
 $$\|w_n-x_{n+1}\|^2\rightarrow 0. $$

  At the same time, it follows from  (\ref{(3.4)}) that
\begin{eqnarray}\label{(3.9)}
\frac{ w_n-x_{n+1}}{\tau_n}-(Bw_n-Bx_{n+1})\in(A+B)x_{n+1},
\end{eqnarray}
 Since $B$ is $\gamma-$cocoercive, we have that $B$ is $\frac{1}{\gamma}$-Lipschitz continuous. Passing to the limit on the  subsequence $\{x_{n_k}\}$ of $\{x_n\}$ converging weakly to $\bar x$ in (\ref{(3.9)}) and taking account that $A+B$ is maximal monotone and thus its graph is weakly-strongly closed, it follows that
$$
0\in(A+B)\bar x,
$$
which means that $\bar x\in \Omega$.
 \vskip 2mm
In view of the fact that the choice of $\bar x$ in $\omega_w(x_n)$ was arbitrary,  namely $\omega_w(x_n)\subset \Omega$, and we conclude from Lemma \ref{Lemma2.5} that $\{x_n\}$ converges weakly to $\bar{x} \in \Omega$.
\vskip 2mm
{\bf  Case 2.}  If the sequence $\{\|x_{n_k}-z\|\}$ does not decrease at infinity in the sense, then there exists a sub-sequence $\{n_k\}$ of $\{n\}$ such that  $\|x_{n_k}-z\|\leq\|x_{{n_k}+1}-z\|$ for all $k\geq 0.$ Furthermore, by Lemma \ref{Lemma2.6}, there exists an integer, non-decreasing sequence $\sigma(n)$ for $n\geq N_1$ (for some $N_1$ large enough) such that $\sigma(n)\rightarrow \infty$ as $n \rightarrow \infty$,
$$\|x_{\sigma(n)}-z\|\leq\|x_{{\sigma(n)}+1}-z\|, \hspace{0.1cm}\|x_{n}-z\|\leq\|x_{{\sigma(n)}+1}-z\|$$
for each $n\geq 0$.
\vskip 2mm
Notice the boundedness of the sequence $\{\|x_{n}-z\|\}$, which implies that there exists  the limit of the sequence $\{\|x_{\sigma(n)}-z\|\}$  and  hence  we conclude that
\begin{eqnarray*}
\lim_{n\rightarrow \infty}(\|x_{\sigma(n)+1}-z\|^2-\|x_{\sigma(n)}-z\|^2)= 0.
\end{eqnarray*}

From (\ref{(3.8)}) with $n$ replaced by $\sigma(n)$, we have
\begin{eqnarray*}
&&\|x_{\sigma(n)+1}-z\|^2-\|x_{\sigma(n)}-z\|^2\\
&\leq&(\theta_{\sigma(n)}-1)(\|x_{\sigma(n)}-z\|^2-\|x_{\sigma(n)-1}-z\|^2)-\theta_{\sigma(n)}(1-\theta_{\sigma(n)})\|x_{\sigma(n)}-x_{\sigma(n)-1}\|^2\\
&&-2\gamma\tau_{\sigma(n)}\Big\|Bw_{\sigma(n)}-Bz-\frac{w_{\sigma(n)}-x_{\sigma(n)+1}}{2\gamma}\Big\|^2+(\frac{\tau_{\sigma(n)}}{2\gamma}-1)\|w_{\sigma(n)}-x_{\sigma(n)+1}\|^2.
\end{eqnarray*}

By a similar argument to {\bf Case 1}, we obtain
 $$\Big\|Bw_{\sigma(n)}-Bz-\frac{w_{\sigma(n)}-x_{\sigma(n)+1}}{2\gamma}\Big\|^2\rightarrow 0; $$
 and
 $$\|w_{\sigma(n)}-x_{\sigma(n)+1}\|^2\rightarrow 0. $$

 These together with (\ref{(3.4)}) implies that
 \begin{eqnarray*}
\frac{ w_{\sigma(n)}-x_{{\sigma(n)}+1}}{\tau_{\sigma(n)}}-(Bw_{\sigma(n)}-Bx_{{\sigma(n)}+1})\in(A+B)x_{{\sigma(n)}+1}.
\end{eqnarray*}

Since the sequence $\{x_n\}$ is bounded, let $\bar{x}\in \omega_w(x_n)$ and $\{x_{\sigma(n)}\}$
be a subsequence of $\{x_n\}$ weakly converging to $\bar x$. Passing to the limit on the  subsequence $\{x_{\sigma(n)}\}$ of $\{x_n\}$ converging weakly to $\bar x$ in the above inequality and taking account that $A+B$ is maximal monotone and thus its graph is weakly-strongly closed, it follows that
$$
0\in(A+B)\bar x,
$$
which means that $\bar x\in \Omega$. In view of the fact that the choice of $\bar x$ in $\omega_w(x_n)$ was arbitrary,  namely $\omega_w(x_n)\subset \Omega$, and we conclude from Lemma \ref{Lemma2.5} that $\{x_n\}$ converges weakly to $\bar{x} \in \Omega$.
\vskip1mm
  This completes the proof.
 \end{proof}

Next we prove the strong convergence of Algorithm 3.2.
\vskip 2mm
%{\bf Theorem 3.6}
\begin{theorem}\label{Theorem3.6}
If the assumptions ${\bf C1}-{\bf C2}$ are satisfied and $\tau_n\in(\epsilon,2\gamma-\epsilon)$ for some given $\epsilon>0$ small enough, we obtain the strong convergence result, namely the sequence $\{x_n\}$  generated by Algorithm 3.2 converges in norm to $z =P_\Omega(0)$
(i.e., the minimum-norm element of the solution set $\Omega$).
\end{theorem}
\begin{proof}
To illustrate the result clearly, the following three situations should be discussed: (I).\hspace{0.1cm} $\theta_n\in (0,1)$; (II).\hspace{0.1cm}$\theta_n\equiv 0$ and  (III).\hspace{0.1cm}$\theta_n\equiv1$, respectively.
\vskip 2mm
{\bf (I).} First we consider the strong convergence under the assumption $\theta_n\in (0,1)$.
\vskip 2mm
Similar to  the previous proof of weak convergence, we begin by showing the boundedness of the sequence $\{x_n\}$.
%\begin{itemize}
%\item[(i)] $\{x_n\}$ is bounded,
%\item[(ii)] $\{x_n\}$ is asymptotically regular, i.e., $\|x_{n+1}-x_n\|\to 0$,
%\item[(iii)] $\omega_w(x_n)\subset \Omega$.
%\end{itemize}
To see this, we denote $z_n=J_{\tau_n}^A(I-\tau_nB)w_n$ and we use the projection $z:=P_\Omega(0)$ in a similar way to
the proof of (\ref{(3.7)}) and (\ref{(3.8)})  of Theorem \ref{Theorem3.5} to get  that
\begin{eqnarray}\label{(3.10)}
\nonumber\|z_n-z\|^2&\leq&\|w_n-z\|^2-\|w_n-z_n\|^2-2\langle w_n-z,\tau_n (Bw_n- Bz)\rangle\\\nonumber
&&+2 \tau_n\langle w_n-z_n,Bw_n- Bz\rangle\\
&\leq&\|w_n-z\|^2+(\frac{\tau_n}{2\gamma}-1)\|w_n-z_n\|^2-2\gamma\tau_n\Big\|Bw_n-Bz-\frac{w_n-z_n}{2\gamma}\Big\|^2,
\end{eqnarray}
hence one can see $\|z_n-z\|\leq\|w_n-z\|$.
\vskip 2mm
It turns out from (\ref{(3.3)}) and (\ref{(3.5)}) that
\begin{eqnarray*}
\|x_{n+1}-z\|&=&\|(1-\alpha_n-\beta_n)w_n+\alpha_nz_n-z\|\\
&=&\|(1-\alpha_n-\beta_n)(w_n-z)+\alpha_n(z_n-z)+\beta_n(-z)\|\\
&\leq& (1-\alpha_n-\beta_n)\|w_n-z\|+\alpha_n\|z_n-z\|+\beta_n\|z\|\\
&\le&(1-\beta_n)[\theta_n\|x_n-z\|+(1-\theta_n)\|x_{n-1}-z\|]+\beta_n\|z\|\\
&\le&(1-\beta_n)(\max\{\|x_n-z\|,\|x_{n-1}-z\|\})+\beta_n\|z\|\\
&\le& \cdots\\
&\le&\max\{\|x_0-z\|,\|x_1-z\|,\|z\|\},
\end{eqnarray*}
which implies that the sequence $\{x_n\}$ is bounded, and so are the sequences $\{w_n\}$, $\{z_n\}$.
\vskip 2mm
%To prove (ii) we we set $v_n=(1-\alpha_n)y_n+\alpha_nz_n$; then  $x_{n+1}=(1-\gamma_n)v_n-\gamma_n\alpha_n(y_n-z_n)$.
%We can get
%\begin{eqnarray*}
%\nonumber
%\|v_n-z\|&=& \|(1-\alpha_n)y_n+\alpha_nz_n-z\| \\ \nonumber
%&\leq& (1-\alpha_n)\|y_n-z\|+\alpha_n\|z_n-z\|\\
%&\leq& \|y_n-z\|,
%\end{eqnarray*}
%and
%\begin{eqnarray}%%(3.15)
%\nonumber
%\|x_{n+1}-y_n\|&=&\|(1-\alpha_n-\gamma_n)y_n+\alpha_nz_n-y_n\| \\ \nonumber
%&=&\|\alpha_n(z_n-y_n)-\gamma_ny_n\| \\
%&\leq&\alpha_n\|z_n-y_n\|+\gamma_n\|y_n\|.
%\end{eqnarray}
Applying the identity (\ref{(2.1)}) we deduce that
\begin{eqnarray}\label{(3.11)}
\|x_{n+1}-z\|^2&=&\|(1-\alpha_n-\beta_n)w_n+\alpha_nz_n-z\|^2 \nonumber\\
&=&\|(1-\alpha_n-\beta_n)(w_n-z)+\alpha_n(z_n-z)+\beta_n(-z)\|^2 \nonumber\\
&\leq&(1-\alpha_n-\beta_n)\|w_n-z\|^2+\alpha_n\|z_n-z\|^2+\beta_n\|z\|^2 \nonumber\\
&& -(1-\alpha_n-\beta_n)\alpha_n\|z_n-w_n\|^2.
\end{eqnarray}
Substituting (\ref{(3.10)}) into (\ref{(3.11)}) and after some manipulations,
we obtain
\begin{eqnarray*}
\|x_{n+1}-z\|^2&\leq&(1-\alpha_n-\beta_n)\|w_n-z\|^2+\alpha_n[\|w_n-z\|^2+(\frac{\tau_n}{2\gamma}-1)\|w_n-z_n\|^2\nonumber\\
&&-2\gamma\tau_n\Big\|Bw_n-Bz-\frac{w_n-z_n}{2\gamma}\Big\|^2]+\beta_n\|z\|^2 \nonumber\\
&& -(1-\alpha_n-\beta_n)\alpha_n\|z_n-w_n\|^2.
\end{eqnarray*}
 Combining (\ref{(3.6)})  we have
\begin{eqnarray}\label{(3.12)}
\|x_{n+1}-z\|^2&\leq&(1-\beta_n)\|w_n-z\|^2++\beta_n\|z\|^2-(1-\alpha_n-\beta_n)\alpha_n\|z_n-w_n\|^2\nonumber\\
&&+\alpha_n[(\frac{\tau_n}{2\gamma}-1)\|w_n-z_n\|^2-2\gamma\tau_n\Big\|Bw_n-Bz-\frac{w_n-z_n}{2\gamma}\Big\|^2] \nonumber\\
&=&(1-\beta_n)[\theta_n\|x_n-z\|^2+(1-\theta_n)\|x_{n-1}-z\|^2-\theta_n(1-\theta_n)\|x_n-x_{n-1}\|^2]\nonumber\\
&&+\beta_n\|z\|^2\nonumber+\alpha_n(\frac{\tau_n}{2\gamma}-2+\alpha_n+\beta_n)\|w_n-z_n\|^2\nonumber\\
&&-2\gamma\tau_n\alpha_n\Big\|Bw_n-Bz-\frac{w_n-z_n}{2\gamma}\Big\|^2.
\end{eqnarray}
\vskip 2mm
 We explain the strong convergence under two situations, respectively.
\vskip 2mm
{\bf Case 1.} There exists an integer $N_0\geq 0$ such that  $\|x_{n+1}-z\|\leq\|x_n-z\|$ for all $n\geq N_0$. Then  there exists the limit of the sequence $\{\|x_n-z\|\}$, denoted by $l=\lim_{n\rightarrow}\|x_n-z\|^2$, and so
 $$\lim_{n\rightarrow\infty}(\|x_n-z\|^2-\|x_{n+1}-z\|^2)=0.$$
In addition, we have
  \begin{eqnarray*}
\sum_{n=0}^\infty(\|x_{n+1}-z\|^2-\|x_n-z\|^2)=\lim_{n\rightarrow \infty}(\|x_{n+1}-z\|-\|x_0-z\|)<\infty
\end{eqnarray*}
and therefore from (\ref{(3.12)}) we obtain
\begin{eqnarray*}
&&\alpha_n(2-\frac{\tau_n}{2\gamma}-\alpha_n-\beta_n)\|w_n-z_n\|^2+2\gamma\tau_n\alpha_n\Big\|Bw_n-Bz-\frac{w_n-z_n}{2\gamma}\Big\|^2\\
&&+\theta_n(1-\theta_n)(1-\beta_n)\|x_n-x_{n-1}\|^2 \\
&\leq&(1-\beta_n)[\theta_n\|x_n-z\|^2+(1-\theta_n)\|x_{n-1}-z\|^2]+\beta_n\|z\|^2-\|x_{n+1}-z\|^2\nonumber\\
&\leq&\theta_n\|x_n-z\|^2+(1-\theta_n)\|x_{n-1}-z\|^2+\beta_n\|z\|^2-\|x_{n+1}-z\|^2\nonumber\\
&=&\|x_n-z\|^2-\|x_{n+1}-z\|^2+(1-\theta_n)(\|x_{n-1}-z\|^2-\|x_n-z\|^2)+\beta_n\|z\|^2,
\end{eqnarray*}
and then from the condition $\beta_n\rightarrow 0$, we get
\begin{eqnarray*}
\lim_{n\rightarrow\infty}\theta_n(1-\theta_n)(1-\beta_n)\|x_n-x_{n-1}\|^2=0;\\
\lim_{n\rightarrow\infty}\alpha_n(2-\frac{\tau_n}{2\gamma}-\alpha_n-\beta_n)\|w_n-z_n\|^2=0;\\
\lim_{n\rightarrow\infty}\gamma\tau_n\alpha_n\Big\|Bw_n-Bz-\frac{w_n-z_n}{2\gamma}\Big\|^2=0.
\end{eqnarray*}
Notice  $\alpha_n(2-\frac{\tau_n}{2\gamma}-\alpha_n-\beta_n)= \alpha_n(1-\alpha_n-\beta_n)+\alpha_n(1-\frac{\tau_n}{2\gamma})$ 
and notice the assumptions on the parameters $\theta_n,\beta_n,\tau_n$ and $\alpha_n$, hence we have
\begin{eqnarray*}
\lim_{n\rightarrow\infty}\|x_n-x_{n-1}\|^2=0;\hspace{0.2cm}
\lim_{n\rightarrow\infty}
\|w_n-z_n\|^2=0;\\
\lim_{n\rightarrow\infty}
\Big\|Bw_n-Bz-\frac{w_n-z_n}{2\gamma}\Big\|^2=0,
\end{eqnarray*}
which imply that $\|w_n-x_n\|=|1-\theta_n| \cdot \|x_n-x_{n-1}\|\rightarrow 0$ and  $\|x_{n+1}-w_n\|\leq\alpha_n\|z_n-w_n\|+\beta_n\|w_n\|\to 0$ as $n\rightarrow \infty$, and then
$$\|x_{n+1}-x_n\|\le \|x_{n+1}-w_n\|+\|w_n-x_n\| \to 0, n\rightarrow \infty$$
This proves the asymptotic regularity of $\{x_n\}$.
\vskip 1mm
By repeating the relevant part of the proof of Theorem \ref{Theorem3.5}, we get $\omega_w(x_n)\subset \Omega$.
\vskip 1mm
It is now at the position to prove the sequence $\{x_n\}$ strongly converges to $z=P_\Omega(0)$.
\vskip 1mm
To this end, we rewrite (\ref{(3.5)}) as $x_{n+1}=(1-\beta_n)v_n+\beta_n\alpha_n(z_n-w_n)$, where $v_n=(1-\alpha_n)w_n+\alpha_nz_n$,
and making use of the inequality $\|u+v\|^2\le \|u\|^2+2\langle v,u+v\rangle$ which holds for all $u,v\in H$,
we get
\begin{eqnarray*}
\|x_{n+1}-z\|^2&=&\|(1-\beta_n)(v_n-z)+\beta_n(\alpha_n(z_n-w_n)-z)\|^2  \nonumber\\
&\leq& (1-\beta_n)^2\|v_n-z\|^2+2\beta_n\Big\langle \alpha_n(z_n-w_n)-z,x_{n+1}-z\Big\rangle.
\end{eqnarray*}
It follows from (\ref{(2.1)}) that
\begin{eqnarray*}
\|v_n-z\|^2=(1-\alpha_n)\|w_n-z\|^2+\alpha_n\|z_n-z\|^2-\alpha_n(1-\alpha_n)\|z_n-w_n\|^2,
\end{eqnarray*}
and then
\begin{eqnarray*}
\|x_{n+1}-z\|^2&\leq& (1-\beta_n)^2[(1-\alpha_n)\|w_n-z\|^2+\alpha_n\|z_n-z\|^2-\alpha_n(1-\alpha_n)\|z_n-w_n\|^2]\nonumber\\
&&+2\beta_n\Big\langle \alpha_n(z_n-w_n)-z,x_{n+1}-z\Big\rangle.
\end{eqnarray*}
 Notice that $\|z_n-z\|^2\leq\|w_n-z\|^2$ from (\ref{(3.10)}), and notice $(1-\beta_n)^2<1-\beta_n $, hence we obtain
 \begin{eqnarray*}
\|x_{n+1}-z\|^2&\leq& (1-\beta_n)\|w_n-z\|^2-\alpha_n(1-\alpha_n)(1-\beta_n)^2\|z_n-w_n\|^2\nonumber\\
&&+2\beta_n\Big\langle \alpha_n(z_n-w_n)-z,x_{n+1}-z\Big\rangle
\end{eqnarray*}
Submitting (\ref{(3.6)}) into the above inequality, we have
\begin{eqnarray}\label{(3.13)}
 \|x_{n+1}-z\|^2&\leq(1-\beta_n)[\theta_n\|x_n-z\|^2+(1-\theta_n)\|x_{n-1}-z\|^2-\theta_n(1-\theta_n)\|x_n-x_{n-1}\|^2]\nonumber\\
&-\alpha_n(1-\alpha_n)(1-\beta_n)^2\|z_n-w_n\|^2+2\beta_n\Big\langle \alpha_n(z_n-w_n)-z,x_{n+1}-z\Big\rangle,
\end{eqnarray}
and then noticing the $\|x_{n+1}-z\|\leq\|x_n-z\|$ for all $n\geq N_0+1$, we have
 \begin{eqnarray}\label{(3.14)}
 \|x_{n+1}-z\|^2&\leq& (1-\beta_n)\|x_{n-1}-z\|^2-\theta_n(1-\theta_n)(1-\beta_n)\|x_n-x_{n-1}\|^2\nonumber\\
&&-\alpha_n(1-\alpha_n)(1-\beta_n)^2\|z_n-w_n\|^2+2\beta_n\Big\langle \alpha_n(z_n-w_n)-z,x_{n+1}-z\Big\rangle  \nonumber\\
&\leq& (1-\beta_n)\|x_{n-1}-z\|^2+2\beta_n\Big\langle \alpha_n(z_n-w_n)-z,x_{n+1}-z\Big\rangle \nonumber\\
&=&(1-\beta_n)\|x_{n}-z\|^2-(1-\beta_n)\|x_{n}-z\|^2+(1-\beta_n)\|x_{n-1}-z\|^2\nonumber\\
  &&+2\beta_n\Big\langle \alpha_n(z_n-w_n)-z,x_{n+1}-z\Big\rangle\nonumber\\
 &=&(1-\beta_n)\|x_{n}-z\|^2+(1-\beta_n)(\|x_{n-1}-z\|^2-\|x_{n}-z\|^2)\nonumber\\
  &&+2\beta_n\Big\langle \alpha_n(z_n-w_n)-z,x_{n+1}-z\Big\rangle\nonumber\\
&=&(1-\beta_n)a_n+\beta_nb_n+c_n,
\end{eqnarray}
where $a_n=\|x_n-z\|^2,\quad b_n=2[\alpha_n\langle z_n-w_n,x_{n+1}-z\rangle+\langle-z,x_{n+1}-z\rangle\,]$ and $c_n=(1-\beta_n)(\|x_{n-1}-z\|^2-\|x_{n}-z\|^2)$.
\vskip 1mm
Since $\omega_w(x_n)\subset \Omega$ and $z=P_{\Omega}(0)$,  which implies from  (\ref{(2.6)}) that
$\langle -z,q-z\rangle\le 0$ for all $q\in\Omega$, we deduce that
\begin{equation}\label{(3.15)}
\limsup_{n\rightarrow\infty}\langle -z,x_{n+1}-z\rangle=\max_{q\in \omega_w(x_n)}\langle -z,q-z\rangle\leq 0.
\end{equation}
Since $\|w_n-z_n\|\rightarrow 0$, $ \overline{\lim}_{n\rightarrow \infty}\alpha_n<1$ and $\|x_{n+1}-z\|$ is  bounded, combining (\ref{(3.15)}) implies that
\begin{align*}
\limsup_{n\to\infty}b_n
&=\limsup_{n\to\infty}2\{\alpha_n\langle z_n-w_n,x_{n+1}-z\rangle+\langle-z,x_{n+1}-z\rangle\}\\
&=\limsup_{n\to\infty}2\langle-z,x_{n+1}-z\rangle\le 0.
\end{align*}
In addition, we have
$$\sum_{n=1}^{\infty}c_n\leq\sum_{n=1}^\infty(\|x_{n-1}-z\|^2-\|x_n-z\|^2)<\infty.$$
These enable us to apply Lemma \ref{Lemma2.4} to (\ref{(3.14)}) to obtain that $a_n\to 0$.
Namely, $x_n\to z$ in norm and the proof of {\bf Case 1} is complete.

\vskip 1mm
{\bf Case 2.}  If the sequence $\{\|x_{n}-z\|\}$ does not decrease at infinity in the sense that there exists a sub-sequence $\{n_k\}$ of $\{n\}$ such that  $\|x_{n_k}-z\|\leq\|x_{{n_k}+1}-z\|$ for all $k\geq 0.$ Furthermore, by Lemma \ref{Lemma2.6}, there exists an integer, non-decreasing sequence $\sigma(n)$ for $n\geq N_1$ (for some $N_1$ large enough) such that $\sigma(n)\rightarrow \infty$ as $n \rightarrow \infty$,
$$\|x_{\sigma(n)}-z\|\leq\|x_{{\sigma(n)}+1}-z\|, \hspace{0.1cm}\|x_{n}-z\|\leq\|x_{{\sigma(n)}+1}-z\|$$
for each $n\geq 0$.
\vskip 2mm
Notice the boundedness of the sequence $\{\|x_{n}-z\|\}$, which implies that there exists  the limit of the sequence $\{\|x_{\sigma(n)}-z\|\}$  and  hence  we conclude that
\begin{eqnarray*}
\lim_{n\rightarrow \infty}(\|x_{\sigma(n)+1}-z\|^2-\|x_{\sigma(n)}-z\|^2)= 0.
\end{eqnarray*}
As a matter of fact, observe that (\ref{(3.12)}) holds for each $n$, so from (\ref{(3.12)}) with $n$ replaced by $\sigma(n)$ and using the relation $\|x_{\sigma(n)}-z\|^2\leq\|x_{\sigma(n)+1}-z\|^2$, we have
\begin{eqnarray*}
&&\theta_{\sigma(n)}(1-\theta_{\sigma(n)})(1-\beta_{\sigma(n)})\|x_{\sigma(n)}-x_{{\sigma(n)}-1}\|^2+\alpha_{\sigma(n)}(2-\frac{\tau_{\sigma(n)}}{2\gamma}-\alpha_{\sigma(n)}-\beta_{\sigma(n)})\|w_{\sigma(n)}-z_{\sigma(n)}\|^2\\
&&+2\gamma\tau_{\sigma(n)}\alpha_{\sigma(n)}\Big\|Bw_{\sigma(n)}-Bz-\frac{w_{\sigma(n)}-z_{\sigma(n)}}{2\gamma}\Big\|^2\\
&\leq&(1-\beta_{\sigma(n)})[\theta_{\sigma(n)}\|x_{\sigma(n)}-z\|^2+(1-\theta_{\sigma(n)})\|x_{{\sigma(n)}-1}-z\|^2]-\|x_{{\sigma(n)}+1}-z\|^2+\beta_{\sigma(n)}\|z\|^2\nonumber\\
&\leq&(1-\beta_{\sigma(n)})\|x_{\sigma(n)}-z\|^2-\|x_{{\sigma(n)}+1}-z\|^2+\beta_{\sigma(n)}\|z\|^2\nonumber\\
&=&\|x_{\sigma(n)}-z\|^2-\|x_{{\sigma(n)}+1}-z\|^2-\beta_{\sigma(n)}(\|x_{\sigma(n)}-z\|^2-\|z\|^2)
\end{eqnarray*}

Notice the assumptions on the parameters $\theta_{\sigma(n)},\beta_{\sigma(n)},\tau_{\sigma(n)}$ and $\alpha_{\sigma(n)}$, Taking the limit by letting $n\to\infty$ yields
\begin{align}
\lim_{n\to\infty}\|w_{\sigma(n)}-z_{\sigma(n)}\|&=0;\label{(3.16)}\\
  \lim_{n\to\infty} \|x_{\sigma(n)}-x_{\sigma(n)-1}\|^2&=0; \label{(3.17)}\\
\lim_{n\to\infty}\Big\|Bw_{\sigma(n)}-Bz-\frac{w_{\sigma(n)}-z_{\sigma(n)}}{2\gamma}\Big\|^2 &=0.\label{(3.18)}
\end{align}

Note that we still have $\|x_{\sigma(n)+1}-x_{\sigma(n)}\|\to 0$  and then that the relations (\ref{(3.16)})-(\ref{(3.18)}) are sufficient to guarantee that $\omega_w(x_{\sigma(n)})\subset \Omega$.
\vskip 2mm
Next we prove $x_{\sigma(n)}\to z$.
\vskip 2mm
Observe that  (\ref{(3.13)}) holds for each $n$.
So replacing $n$ with $\sigma(n)$ in (\ref{(3.13)}) and using
the relation $\|x_{\sigma(n)}-z\|^2\leq\|x_{\sigma(n)+1}-z\|^2$ again for $n\geq N_1$, we obtain
\begin{eqnarray*}
 \nonumber\|x_{\sigma(n)+1}-z\|^2&=&(1-\beta_{\sigma(n)})[\theta_{\sigma(n)}\|x_{\sigma(n)}-z\|^2+(1-\theta_{\sigma(n)})\|x_{\sigma(n)-1}-z\|^2\\ \nonumber
 &&-\theta_{\sigma(n)}(1-\theta_{\sigma(n)})\|x_{\sigma(n)}-x_{\sigma(n)-1}\|^2]\nonumber\\\nonumber
&&-\alpha_{\sigma(n)}(1-\alpha_{\sigma(n)})(1-\beta_{\sigma(n)})^2\|z_{\sigma(n)}-w_{\sigma(n)}\|^2\\\nonumber
&&+2\beta_{\sigma(n)}\Big\langle \alpha_{\sigma(n)}(z_{\sigma(n)}-w_{\sigma(n)})-z,x_{\sigma(n)+1}-z\Big\rangle \\\nonumber
&\leq& (1-\beta_{\sigma(n)})\|x_{\sigma(n)}-z\|^2
+2\beta_{\sigma(n)}\Big\langle \alpha_{\sigma(n)}(z_{\sigma(n)}-w_{\sigma(n)})-z,x_{\sigma(n)+1}-z\Big\rangle,
\end{eqnarray*}
and then we obtain
\begin{eqnarray*}
 \nonumber\beta_{\sigma(n)})\|x_{\sigma(n)}-z\|^2&\leq& \|x_{\sigma(n)}-z\|^2-\|x_{\sigma(n)+1}-z\|^2
+2\beta_{\sigma(n)}\Big\langle \alpha_{\sigma(n)}(z_{\sigma(n)}-w_{\sigma(n)})-z,x_{\sigma(n)+1}-z\Big\rangle\\
&\leq&2\beta_{\sigma(n)}\Big\langle \alpha_{\sigma(n)}(z_{\sigma(n)}-w_{\sigma(n)})-z,x_{\sigma(n)+1}-z\Big\rangle,
\end{eqnarray*}
which means that there exists a constant  $M$ such that $M\ge 2\|x_{n}-z\|$ for all $n$ and 
\begin{eqnarray}\label{(3.19)}
\nonumber\|x_{\sigma(n)}-z\|^2 &\leq& 2\Big\langle \alpha_{\sigma(n)}(z_{\sigma(n)}-w_{\sigma(n)})-z,x_{\sigma(n)+1}-z\Big\rangle\\
&\le& M\|z_{\sigma(n)}-w_{\sigma(n)}\|+2\Big\langle-z,x_{{\sigma(n)}+1}-z\Big\rangle.
\end{eqnarray}
Now since  $\|x_{\sigma(n)+1}-x_{\sigma(n)}\|\to 0$, we have
\begin{align*}
\limsup_{n\to\infty} \langle -z,x_{\sigma(n)+1}-z\rangle &=\limsup_{n\to\infty} \langle -z,x_{\sigma(n)}-z\rangle\\
&=\max_{q\in \omega_w(x_{\sigma(n)})}\langle -z,q-z\rangle\le 0
\end{align*}
by virtue of the facts $z=P_{\Omega}(0)$ and $\omega(x_{\sigma(n)})\subset\Omega$.
\vskip 1mm
Consequently, following from $\|z_{\sigma(n)}-w_{\sigma(n)}\|\to 0$, the relation (\ref{(3.19)}) assures that $x_{\sigma(n)}\to z$, which further implies from Lemma \ref{Lemma2.6} that
$$\|x_n-z\|\le\|x_{\sigma(n)+1}-z\|\le \|x_{\sigma(n)+1}-x_{\sigma(n)}\|+\|x_{\sigma(n)}-z\|\to 0.$$
Namely, $x_n\to z$ in norm, and the proof of Case 2 is complete.
\vskip 1mm

{\bf (II).} Now, it is time to show the strong convergence when $\theta_n\equiv 0$.  Clearly, if $\theta_n=0$, then $w_n=x_{n-1}$ and $x_{n+1}=(1-\alpha_n-\beta_n)x_{n-1}+\alpha_nz_{n-1}$, where $z_{n-1}=J_{\tau_n}^A(I-\tau_nB)x_{n-1}$.
\vskip 1mm
Repeating  the steps from (\ref{(3.10)})-(\ref{(3.12)}), we have the following similar inequality:
\begin{eqnarray}\label{(3.20)}
\|x_{n+1}-z\|^2&\leq&(1-\beta_n)\|x_{n-1}-z\|^2+\beta_n\|z\|^2+\alpha_n(\frac{\tau_n}{2\gamma}-2+\alpha_n+\beta_n)\|x_{n-1}-z_{n-1}\|^2\nonumber\\
&&-2\gamma\tau_n\alpha_n\Big\|Bx_{n-1}-Bz-\frac{x_{n-1}-z_{n-1}}{2\gamma}\Big\|^2.
\end{eqnarray}
Two possible cases will be shown as follows.
\vskip 1mm
{\bf Case 1.} There exists an integer $N_0\geq 0$ such that  $\|x_{n+1}-z\|\leq\|x_n-z\|$ for all $n\geq N_0$. Then  there exists the limit of the sequence $\{\|x_n-z\|\}$, denoted by $l=\lim_{n\rightarrow}\|x_n-z\|^2$, and so
 $$\lim_{n\rightarrow\infty}(\|x_n-z\|^2-\|x_{n+1}-z\|^2)=0.$$
In addition, we have
  \begin{eqnarray*}
\sum_{n=0}^\infty(\|x_{n+1}-z\|^2-\|x_n-z\|^2)=\lim_{n\rightarrow \infty}(\|x_{n+1}-z\|^2-\|x_0-z\|^2)<\infty
\end{eqnarray*}
and therefore from (\ref{(3.20)}) we obtain
\begin{eqnarray*}
&&\alpha_n(2-\frac{\tau_n}{2\gamma}-\alpha_n-\beta_n)\|x_{n-1}-z_{n-1}\|^2+2\gamma\tau_n\alpha_n\Big\|Bx_{n-1}-Bz-\frac{x_{n-1}-z_{n-1}}{2\gamma}\Big\|^2\\
&\leq&(1-\beta_n)\|x_{n-1}-z\|^2+\beta_n\|z\|^2-\|x_{n+1}-z\|^2\nonumber\\
&=&(\|x_{n-1}-z\|^2-\|x_{n}-z\|^2)+ (\|x_{n}-z\|^2-\|x_{n+1}-z\|^2)+\beta_n(\|z\|^2-\|x_{n-1}-z\|^2).
\end{eqnarray*}
Notice the assumptions on the parameters $\alpha_n$, $\beta_n$ and $\tau_n$,  we have
\begin{eqnarray*}
\lim_{n\rightarrow\infty}
\|x_{n-1}-z_{n-1}\|^2=0;
\lim_{n\rightarrow\infty}
\Big\|Bx_{n-1}
-Bz-\frac{x_{n-1}-z_{n-1}}{2\gamma}\Big\|^2=0,
\end{eqnarray*}
which imply that  $\|x_{n+1}-x_{n-1}\|\leq\alpha_n\|z_{n-1}-x_{n-1}\|+\beta_n\|x_{n-1}\|\to 0$ as $n\rightarrow \infty$.
\vskip 1mm
Next we show the asymptotic regularity of $\{x_n\}$.
Indeed, it follows from the relation between  the norm and inner product that
\begin{eqnarray*}
\|x_{n+1}-x_n\|^2&=&\|x_{n+1}-z+z-x_n\|^2\\
&=&\|x_{n+1}-z\|^2+\|z-x_n\|^2+2\langle x_{n+1}-z,z-x_n\rangle\\
&=&\|x_{n+1}-z\|^2+\|z-x_n\|^2+2\langle x_{n+1}-z,z-x_{n+1}+x_{n+1}-x_n\rangle\\
&\leq&\|x_n-z\|^2-\|x_{n+1}-z\|^2+2\|x_{n+1}-z\|\cdot\|x_{n+1}-x_n\|\\
&\leq&\|x_n-z\|^2-\|x_{n+1}-z\|^2+2(M+m)\cdot\|x_{n+1}-x_n\|,
\end{eqnarray*}
where $M$ is a constant such that $M\ge \|x_{n}-z\|$ for all $n$ and $m>0$ is a given constant, which means that
\begin{eqnarray*}
\|x_{n+1}-x_n\|^2-2(M+m)\cdot\|x_{n+1}-x_n\|\leq\|x_n-z\|^2-\|x_{n+1}-z\|^2,
\end{eqnarray*}
 and then
 \begin{eqnarray*}
\sum_{n=0}^\infty[\|x_{n+1}-x_n\|-2(M+m)]\cdot\|x_{n+1}-x_n\|\leq\sum_{n=0}^\infty(\|x_n-z\|^2-\|x_{n+1}-z\|^2)<\infty.
\end{eqnarray*}
Therefore we obtain $[\|x_{n+1}-x_n\|-2(M+m)]\cdot\|x_{n+1}-x_n\|\rightarrow 0$ as $n\rightarrow \infty$.\\

Since $\|x_{n+1}-x_n\|\leq\|x_{n+1}-z\|+\|x_n-z\|\leq 2M$, we have $\|x_{n+1}-x_n\| \to 0$. This proves the asymptotic regularity of $\{x_n\}$.
 \vskip 1mm
 By  repeating the relevant part of the proof of {\bf Case 1} in {\bf (I)}, we get  $\omega_w(x_n)\subset \Omega$ and
  \begin{eqnarray}\label{(3.21)}
 \|x_{n+1}-z\|^2&\leq&(1-\beta_n)\|x_{n-1}-z\|^2-\alpha_n(1-\alpha_n)(1-\beta_n)^2\|z_{n-1}-x_{n-1}\|^2 \nonumber\\
 &&+2\beta_n\Big\langle \alpha_n(z_{n-1}-x_{n-1})-z,x_{n+1}-z\Big\rangle  \nonumber\\
 &\leq& (1-\beta_n)\|x_{n}-z\|^2+(1-\beta_n)(\|x_{n-1}-z\|^2-\|x_{n}-z\|^2)\nonumber\\
  &&+2\beta_n\Big\langle \alpha_n(z_{n-1}-x_{n-1})-z,x_{n+1}-z\Big\rangle\nonumber\\
&=&(1-\beta_n)a_n+\beta_nb_n+c_n,
\end{eqnarray}
where $a_n=\|x_n-z\|^2,\quad b_n=2[\alpha_n\langle z_{n-1}-x_{n-1},x_{n+1}-z\rangle+\langle-z,x_{n+1}-z\rangle\,]$ and $c_n=(1-\beta_n)(\|x_{n-1}-z\|^2-\|x_{n}-z\|^2)$.
\vskip 1mm
The rest of proof  is consistent with  {\bf Case 1} in {\bf (I).} So we get $x_n\to z$ in norm.
\vskip 1mm
{\bf Case 2.}  If the sequence $\{\|x_{n}-z\|\}$ does not decrease at infinity in the sense that there exists a sub-sequence $\{n_k\}$ of $\{n\}$ such that  $\|x_{n_k}-z\|\leq\|x_{{n_k}+1}-z\|$ for all $k\geq 0.$ Furthermore, by Lemma \ref{Lemma2.6}, there exists an integer, non-decreasing sequence $\sigma(n)$ for $n\geq N_1$ (for some $N_1$ large enough) such that $\sigma(n)\rightarrow \infty$ as $n \rightarrow \infty$,
$$\|x_{\sigma(n)}-z\|\leq\|x_{{\sigma(n)}+1}-z\|, \hspace{0.1cm}\|x_{n}-z\|\leq\|x_{{\sigma(n)}+1}-z\|$$
for each $n\geq 0$.
\vskip 2mm
Note that we still have $\|x_{\sigma(n)+1}-x_{\sigma(n)}\|\to 0$.
\vskip 1mm
By  repeating the relevant part of the proof of {\bf Case 2} in {\bf (I)}, for $n\geq N_1$, we have
\begin{eqnarray*}
&&\alpha_{\sigma(n)}(2-\frac{\tau_{\sigma(n)}}{2\gamma}-\alpha_{\sigma(n)}-\beta_{\sigma(n)})\|x_{\sigma(n)-1}-z_{\sigma(n)-1}\|^2\\
&&+2\gamma\tau_{\sigma(n)}\alpha_{\sigma(n)}\Big\|Bx_{\sigma(n)-1}-Bz-\frac{x_{\sigma(n)-1}-z_{\sigma(n)-1}}{2\gamma}\Big\|^2\\
&\leq&(1-\beta_{\sigma(n)})\|x_{{\sigma(n)}-1}-z\|^2-\|x_{{\sigma(n)}+1}-z\|^2+\beta_{\sigma(n)}\|z\|^2\nonumber\\
&\leq&(1-\beta_{\sigma(n)})\|x_{\sigma(n)}-z\|^2-\|x_{{\sigma(n)}+1}-z\|^2+\beta_{\sigma(n)}\|z\|^2\nonumber\\
&=&\|x_{\sigma(n)}-z\|^2-\|x_{{\sigma(n)}+1}-z\|^2-\beta_{\sigma(n)}(\|x_{\sigma(n)}-z\|^2-\|z\|^2)
\end{eqnarray*}
Notice the assumptions on the parameters $\alpha_{\sigma(n)}$, $\beta_{\sigma(n)}$ and $\tau_{\sigma(n)}$, taking the limit by letting $n\to\infty$ yields
\begin{align*}
\lim_{n\to\infty}\|x_{\sigma(n)-1}-z_{\sigma(n)-1}\|&=0;\\
  \lim_{n\to\infty}\Big\|Bx_{\sigma(n)-1}-Bz-\frac{x_{\sigma(n)-1}-z_{\sigma(n)-1}}{2\gamma}\Big\|^2 &=0,
\end{align*}
and  that these relations  are sufficient to guarantee that $\omega_w(x_{\sigma(n)})\subset \Omega$.
\vskip 1mm
Observe that (\ref{(3.21)}) holds for each $n$.
So replacing $n$ with $\sigma(n)$ in (\ref{(3.21)}) and using
the relation $\|x_{\sigma(n)}-z\|^2<\|x_{\sigma(n)+1}-z\|^2$ again for $n\geq N_1$, we obtain
\begin{eqnarray*}
 \nonumber\|x_{\sigma(n)+1}-z\|^2&\leq&(1-\beta_{\sigma(n)})\|x_{\sigma(n)-1}-z\|^2-\alpha_{\sigma(n)}(1-\alpha_{\sigma(n)})(1-\beta_{\sigma(n)})^2\|z_{\sigma(n)-1}-x_{\sigma(n)-1}\|^2\\\nonumber
&&+2\beta_{\sigma(n)}\langle \alpha_{\sigma(n)}(z_{\sigma(n)-1}-w_{\sigma(n)-1})-z,x_{\sigma(n)+1}-z\rangle \\\nonumber
&\leq& (1-\beta_{\sigma(n)})\|x_{\sigma(n)}-z\|^2
+2\beta_{\sigma(n)}\langle \alpha_{\sigma(n)}(z_{\sigma(n)-1}-x_{\sigma(n)-1})-z,x_{\sigma(n)+1}-z\rangle,
\end{eqnarray*}
The rest of proof  is consistent with  {\bf Case 2} in {\bf (I).} So we get $x_n\to z$ in norm.
\vskip 2mm
{\bf (III).} Finally we consider the situation $\theta_n\equiv 1$. It is obvious that if $\theta_n=1$, then $w_n=x_{n}$ and $x_{n+1}=(1-\alpha_n-\beta_n)x_{n}+\alpha_nz_{n}$, where $z_{n}=J_{\tau_n}^A(I-\tau_nB)x_{n}$.
\vskip 1mm
Repeating  the steps from (\ref{(3.10)})-(\ref{(3.12)}), we have the following similar inequality
\begin{eqnarray*}
\|x_{n+1}-z\|^2&\leq&(1-\beta_n)\|x_{n}-z\|^2+\beta_n\|z\|^2+\alpha_n(\frac{\tau_n}{2\gamma}-2+\alpha_n+\beta_n)\|x_{n}-z_{n}\|^2\nonumber\\
&&-2\gamma\tau_n\alpha_n\Big\|Bx_{n}-Bz-\frac{x_{n}-z_{n}}{2\gamma}\Big\|^2\\
&\leq&(1-\beta_n)\|x_{n}-z\|^2+\beta_n\|z\|^2\\
&\leq&\max\{\|x_{n}-z\|^2,\|z\|^2\},
\end{eqnarray*}
which means that $\|x_{n}-z\|^2$ is bounded.
Similar to the proof of the above situations, two possible cases will be shown as follows.
\vskip 1mm
Two possible cases will be shown as follows.
\vskip 1mm
{\bf Case 1.} There exists an integer $N_0\geq 0$ such that  $\|x_{n+1}-z\|\leq\|x_n-z\|$ for all $n\geq N_0$. Then  there exists the limit of the sequence $\{\|x_n-z\|\}$, denoted by $l=\lim_{n\rightarrow}\|x_n-z\|^2$, and so
 $$\lim_{n\rightarrow\infty}(\|x_n-z\|^2-\|x_{n+1}-z\|^2)=0.$$
In addition, we have
  \begin{eqnarray*}
\sum_{n=0}^\infty(\|x_{n+1}-z\|^2-\|x_n-z\|^2)=\lim_{n\rightarrow \infty}(\|x_{n+1}-z\|-\|x_0-z\|)<\infty.
\end{eqnarray*}
By  repeating the relevant part of the proof of {\bf Case 1} in {\bf (I)}, we have
\begin{eqnarray*}
\lim_{n\rightarrow\infty}
\|x_{n}-z_{n}\|^2=0;
\lim_{n\rightarrow\infty}
\Big\|Bx_n-Bz-\frac{x_{n}-z_{n}}{2\gamma}\Big\|^2=0,
\end{eqnarray*}
which imply that  $\|x_{n+1}-x_{n}\|\leq\alpha_n\|z_{n}-x_{n}\|+\beta_n\|x_{n}\|\to 0$ as $n\rightarrow \infty$.
\vskip 1mm
The rest of proof  is consistent with  {\bf Case 1} in {\bf (I).} So we get $x_n\to z$ in norm.
\vskip 1mm
{\bf Case 2.}  If the sequence $\{\|x_{n}-z\|\}$ does not decrease at infinity in the sense that there exists a sub-sequence $\{n_k\}$ of $\{n\}$ such that  $\|x_{n_k}-z\|\leq\|x_{{n_k}+1}-z\|$ for all $k\geq 0.$ Furthermore, by Lemma \ref{Lemma2.6}, there exists an integer, non-decreasing sequence $\sigma(n)$ for $n\geq N_1$ (for some $N_1$ large enough) such that $\sigma(n)\rightarrow \infty$ as $n \rightarrow \infty$,
$$\|x_{\sigma(n)}-z\|\leq\|x_{{\sigma(n)}+1}-z\|, \hspace{0.1cm}\|x_{n}-z\|\leq\|x_{{\sigma(n)}+1}-z\|$$
for each $n\geq 0$.

Note that we still have $\|x_{\sigma(n)+1}-x_{\sigma(n)}\|\to 0$. The rest of proof  is consistent with  {\bf Case 1} in {\bf (I).} So we get $x_n\to z$ in norm.

\vskip 1mm
This completes the proof.
\end{proof}
\vskip 2mm
%\subsection{ Convergence Rate Analysis for Algorithms}
\begin{remark}\label{Remark3.7}
It is easy to see that  the convergence of our algorithms still holds even without the following  condition:
 $$\sum_{n=1}^\infty\theta_n\|x_n-x_{n-1}\|^2<\infty.$$
The above assumption is not necessary  at all in our cases. To some extents,  our  inertial-like algorithms seem to have two merits:\vskip 1mm
(1) Compared with the general  inertial proximal algorithms, we do not need to calculate the values of $\|x_n-x_{n-1}\|$ before choosing the parameters $\theta_n$ in numerical simulations, which make the algorithms  convenient and use-friendly.
\vskip 1mm
(2) Compared with the  general inertial algorithms, the inertial factors $\theta_n$ are chosen in $[0,1]$ with $\theta_n=1$ possible, which are new, natural  and interesting algorithms in some ways. In particular,  under more mild assumptions, our proofs are simpler and different from the others.
\end{remark}

%%%%%%%%%%%%%%%%%%%%%%%%%%%%%%%%%%%%%%%%%%%%%%%%%%%%%%%%%
\section{Applications}
\subsection{Convex Optimization}
 Let $C$ be a nonempty closed and convex subset of $H$ and $f, g$ be two proper, convex and lower
semi-continuous functions. Moreover, assume that $g$ is differentiable with a $1/\gamma-$Lipschitz gradient. With this data, consider the following convex minimizing problem :
\begin{eqnarray} \label{Eq:4.1}
\min_{x\in C}\{f(x)+g(x)\}.
\end{eqnarray}
\vskip 2mm
Denoted by $\Omega=\{x:\min_{x\in C}\{f(x)+g(x)\}$.  Recall that the {\it subdifferential} of $f$ at $x$, denote by $\partial f$:
 \begin{eqnarray*}
\partial f(x):=\{x^*\in C:f(y)-f(x)\geq\langle y-x,x^*\rangle,\,\forall y\in H\}.
\end{eqnarray*}
\vskip 2mm

So, by taking  $A= \partial f$ and $B=\nabla  g$(: the gradient of $g$) in (\ref{(3.2)}), where $A$ and $B$ are maximal monotone operators and $B$ is $\gamma$-cocoercive, we can apply Theorem \ref{Theorem3.5} and Theorem \ref{Theorem3.6} and obtain the following results:
 \vskip 2mm

\begin{theorem}\label{Theorem4.1}
 Let $f$ and $g$ be two proper, convex and lower semi-continuous functions such that $\nabla g$ and $\partial f$  are maximal monotone operators and $\nabla g$ also $\gamma-$cocoercive.  Assume that $\Omega\neq\emptyset$.   Let $\{x_n\}$ and $\{y_n\}$ be any two sequences generated by the following scheme (see, e.g. ):
\begin{equation*}
\begin{cases}
  w_n=x_{n-1}+\theta_n(x_n-x_{n-1}),\\
x_{n+1}=J^{\partial f}_{\tau_n}(w_n-\tau_n \nabla g w_n),\\
\end{cases}
\end{equation*}
If $\tau_n\in(\epsilon, 2\gamma-\epsilon)$  for some given $\epsilon>0$ small enough, and $ \theta_n\in [0,1]$, then the sequences $\{x_n\}$ converges weakly to a point of ${\Omega.}$
\end{theorem}

\begin{theorem}\label{Theorem4.2}
 Let $f$ and $g$ be two proper, convex and lower semi-continuous functions such that $\nabla g$ and $\partial f$  are maximal monotone operators and $\nabla g$ also $\gamma-$cocoercive.  Assume that $\Omega\neq\emptyset$.   Let $\{x_n\}$ and $\{y_n\}$ be any two sequences generated by the following scheme:
\begin{equation*}
\begin{cases}
  w_n=x_{n-1}+\theta_n(x_n-x_{n-1}),\\
x_{n+1}=(1-\alpha_n-\beta_n)w_n+\alpha_nJ^{\partial f}_{\tau_n}(w_n-\tau_n \nabla g w_n),\\
\end{cases}
\end{equation*}
where $\alpha_n,\beta_n,\theta_n$ are satisfying the selection criteria in Algorithm3.2. If $\tau_n\in(\epsilon, 2\gamma-\epsilon)$ for some given $\epsilon>0$ small enough, then the sequences $\{x_n\}$ converges strongly to $z=P_{\Omega}(0)$.
\end{theorem}

\subsection{Variational Inequality Problem}
\vskip 3mm

Let $C$ be a nonempty closed and convex subset of $H$ and $B:H\rightarrow H$ a maximal and $\gamma-$coercive operator.
 \vskip 1mm

Consider the classical {\it variational inequality} (VI) problem of finding a point $x^*\in C$ such that
\begin{eqnarray} \label{Eq:4.2}
\langle Bx^*,z-x^*\rangle\geq 0,\,\,\,\forall z\in C.
\end{eqnarray}

Denote by $\Omega$ the solution set of the problem (VI) (\ref{Eq:4.2}).
\vskip 1mm

So, by taking $Ax:=\{w\in H:\langle w,z-x\rangle\leq 0,\,\forall z\in C\}$ (: the \textit{normal cone} of the set $C$),  the problem (VI) (\ref{Eq:4.2}) is equivalent to finding zeroes of $A+B$ (see e.g., Peypouquet\cite{P2015}).
 \vskip 1mm
We apply our algorithms to the {\it variational inequality} (VI) problem and have the following theorems.
\begin{theorem} \label{Th:4.3}
{\it Let $C$ be a nonempty closed and convex subset of $H$, $B:H\rightarrow H$ a maximal and $\gamma-$coercive operator. Choose $\tau_n\in(\epsilon, 2\gamma-\epsilon)$   for some given $\epsilon>0$ small enough and assume that $\Omega\neq\emptyset$. Construct the sequences $\{x_n\}$ and $\{w_n\}$ as follows:
\begin{equation*}
\begin{cases}
 w_n=x_{n-1}+\theta_n(x_n-x_{n-1}),\\
x_{n+1}=J^A_{\tau_n}(I-\tau_n B)w_n.
\end{cases}
\end{equation*}
If  $\theta_n\in [0,1]$, then the sequences $\{x_n\}$  converges weakly to a point of ${\Omega.}$}
\end{theorem}
\begin{theorem}\label{Theorem4.4}
{\it Let $C$ be a nonempty closed and convex subset of $H$, $B:H\rightarrow H$ be a maximal and $\gamma-$cocoercive  operator.  Choose $\tau_n\in(\epsilon, 2\gamma-\epsilon)$  for some given $\epsilon>0$ small enough and assume that $\Omega\neq\emptyset$. Construct the sequences $\{x_n\}$, $\{w_n\}$ as follows.
\begin{equation*}
\begin{cases}
 w_n=x_{n-1}+\theta_n(x_n-x_{n-1}),\\
x_{n+1}=(1-\alpha_n-\beta_n)w_n+\alpha_nJ^A_{\tau_n}(I-\tau_n B)w_n,
\end{cases}
\end{equation*}
where  $\alpha_n,\beta_n,\theta_n$ are satisfying the selection criteria in Algorithm 3.2.
Then the sequence $\{x_n\}$ converges strongly to $z=P_{\Omega}(0)$.}
\end{theorem}
\vskip 4mm

\section{ Numerical Examples}
In this section, we first present two numerical examples in infinite and finite-dimensional Hilbert spaces to illustrate the applicability, efficiency and stability of Algorithm 3.1 and Algorithm 3.2, and then consider sparse signal recovery  from real-world life in finite dimensional spaces.
In addition, the comparison results with other algorithms  are also described. All the codes are written
in Matlab R2016b and are preformed on an LG dual core personal computer.
\vskip 2mm
{\bf Example 5.1.}In this example,  we take $H=R^3$ with Euclidean norm. Let $A:R^3\rightarrow R^3$ be defined by  $Ax=3x$ and let $B:R^3\rightarrow R^3$ be defined by $Bx=\frac{x}{3}+(-1,2,0)$, $\forall x\in R^3$. We can see that $A$ and $B$ are maximal monotone mappings and $B$ with $\gamma$-cocoercive($0<\gamma\leq 3$), respectively. Indeed, let $x,y\in R^3$, then
\begin{equation*}
\langle Bx-By,x-y\rangle=\langle \frac{x}{3}-\frac{y}{3},x-y\rangle=\frac{1}{3}\|x-y\|^2,
\end{equation*}
while
\begin{equation*}
\|Bx-By\|^2=\frac{1}{9}\|x-y\|^2,
\end{equation*}
which means $B$ is $\frac{1}{3}$-cocoercive. It is not hard to check that $(A+B)^{-1}(0)=(3/10,-6/10,0)$.
\vskip 1mm
 In the simulation process, we choose $x_0=[0.1;-0.2;0.1], x_1=[0.2;0.1;-0.3]$ as two arbitrary initial points.
In order to investigate the change  and tendency of $\{x_n\}$ more clearly, we denote by  $z=(3/10,-6/10,0)$ and  let $\|x_{n+1}-z\|\leq 10^{-5}$ be the stopping criterion. The experimental results are shown in Figs.\ref{Fig:6}-\ref{Fig:7}, where z-axis represents the logarithm  of the third coordinate value for each point.
\begin{figure}[]
  \centering
  \includegraphics[width=5cm]{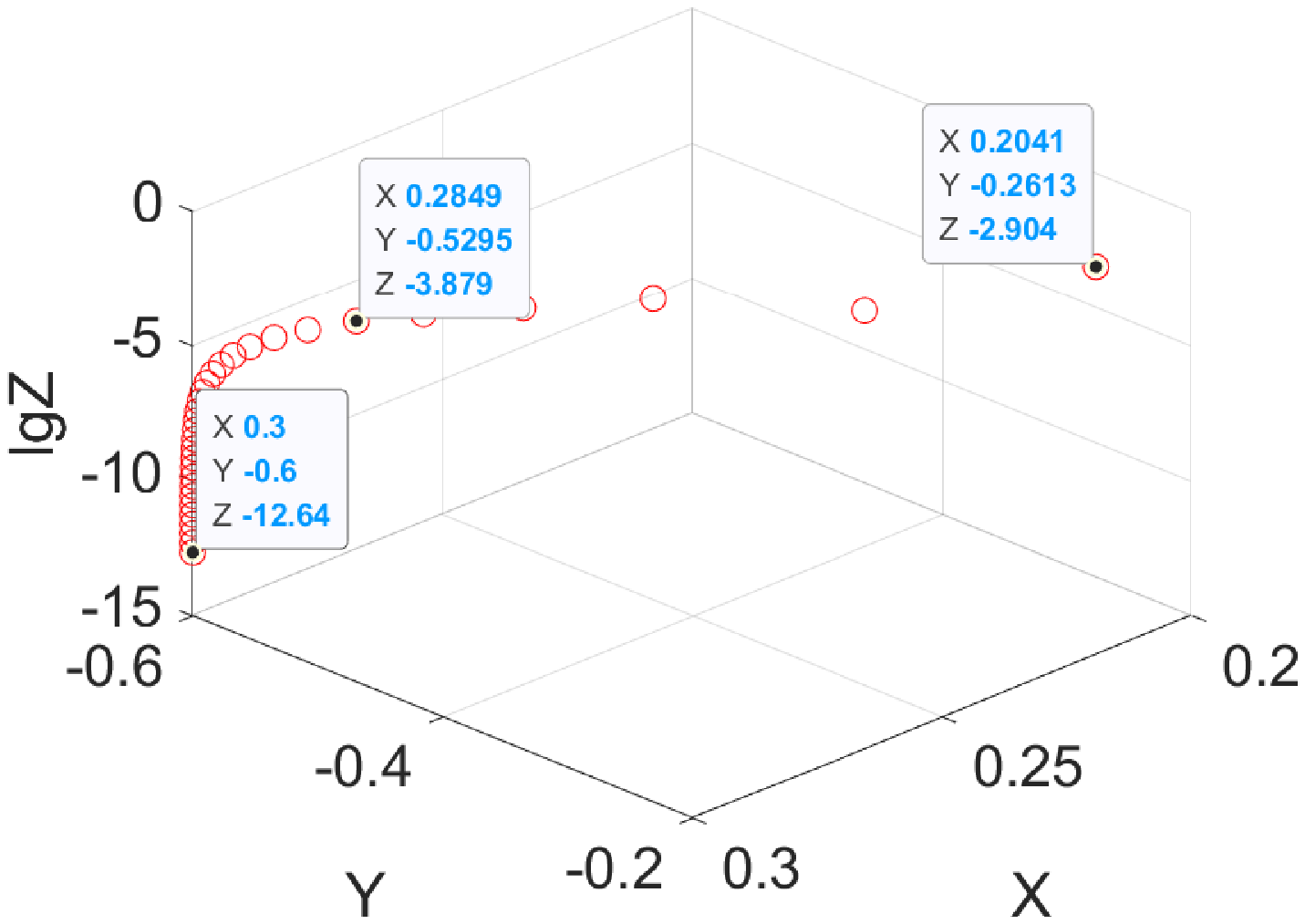}
  \includegraphics[width=5cm]{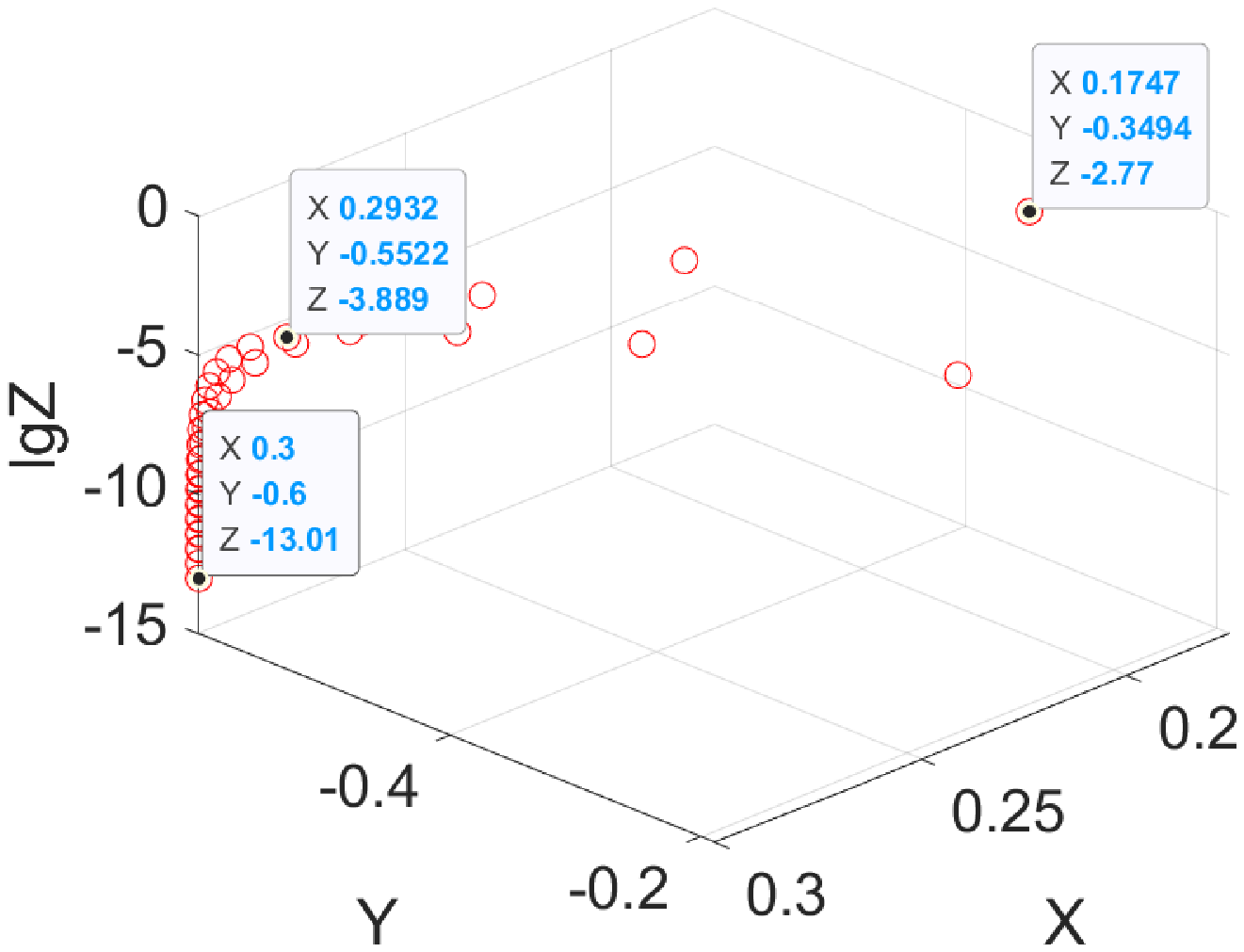}
  \includegraphics[width=5cm]{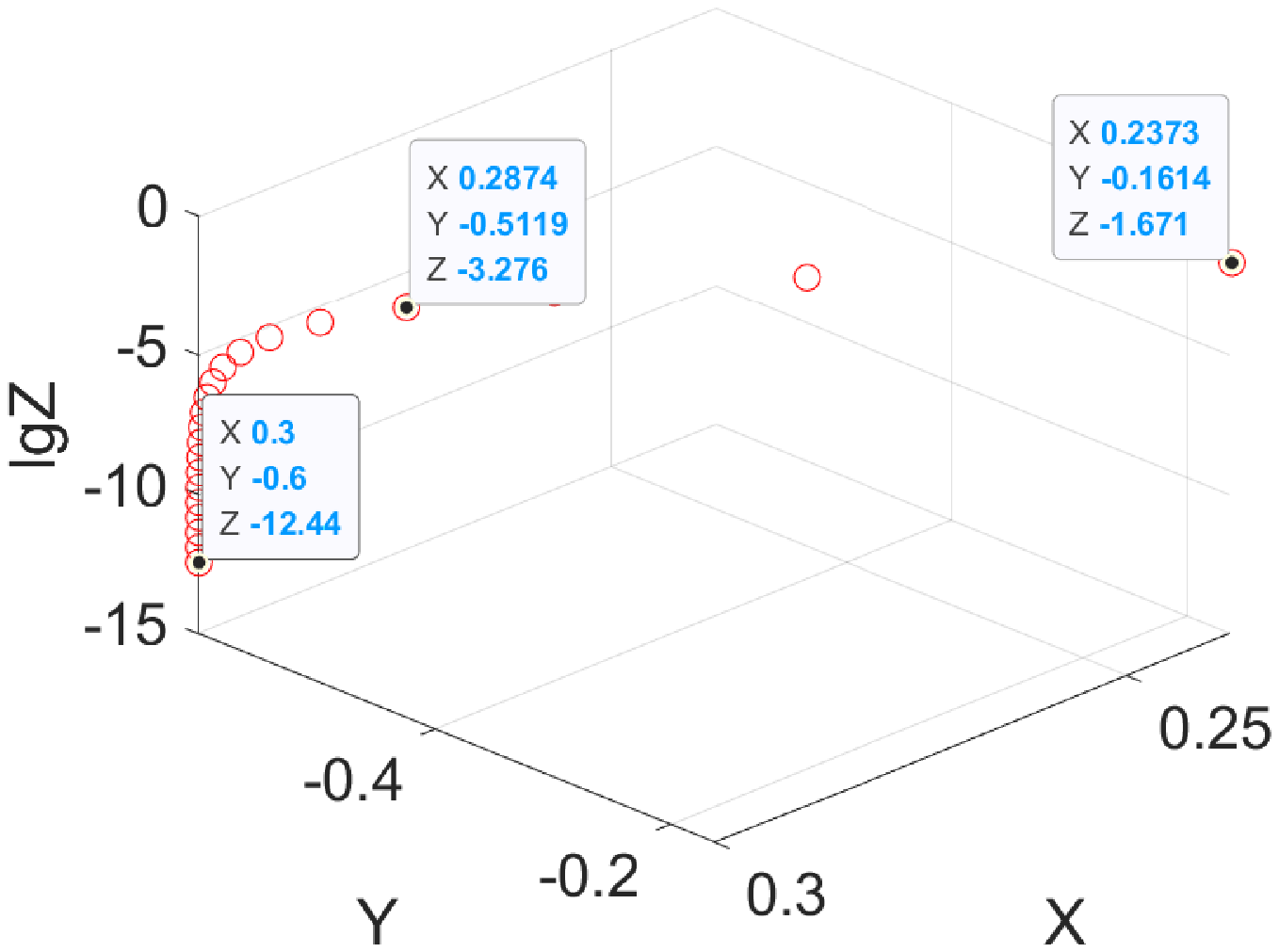}
     \caption{Behaviors of  Algorithm 3.1 for $\theta_n=0.5-1/(n+1)^5,\theta_n=0,\theta_n=1.$}
\label{Fig:6}
\end{figure}
\begin{figure}[]
  \centering
  \includegraphics[width=5cm]{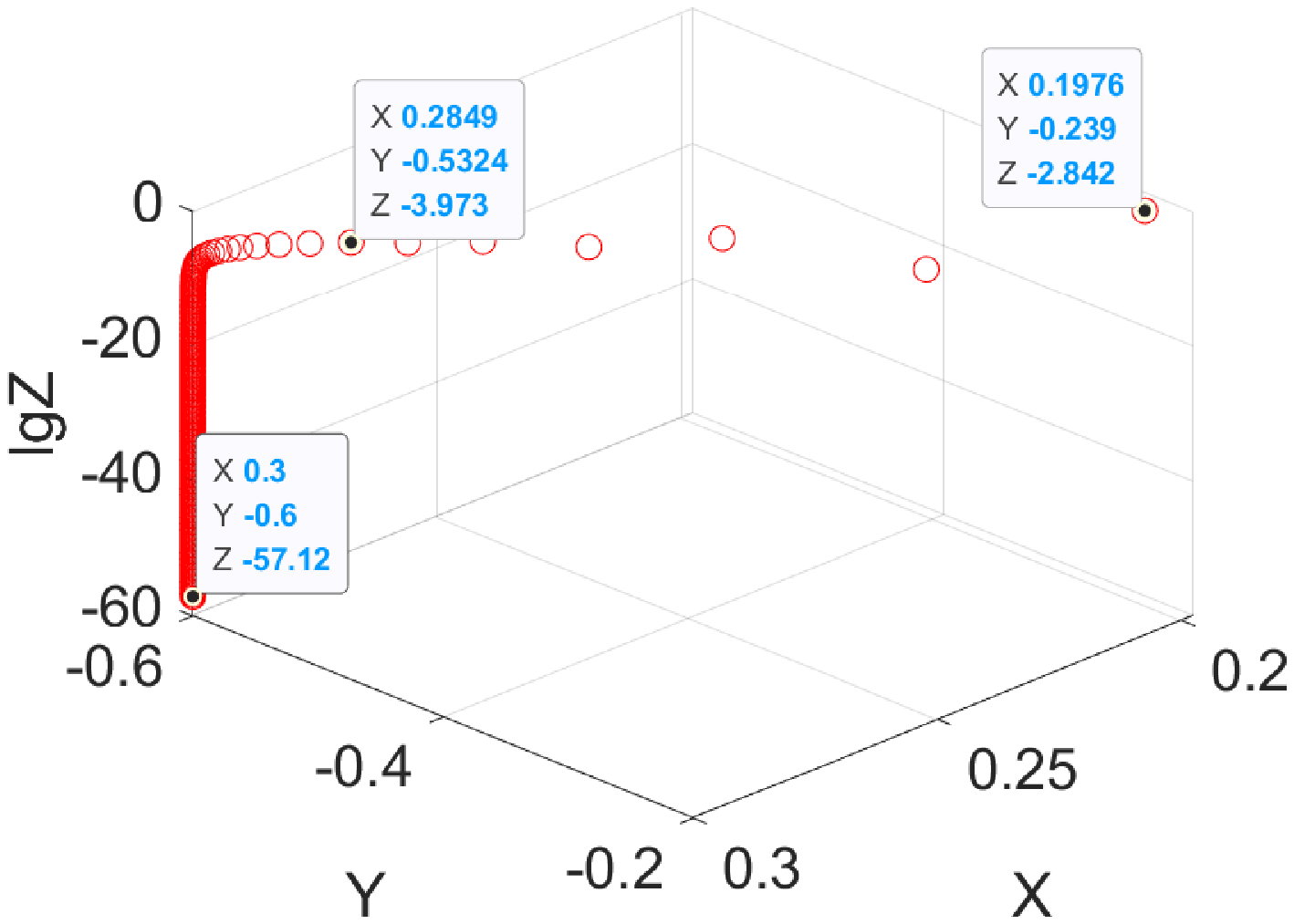}
  \includegraphics[width=5cm]{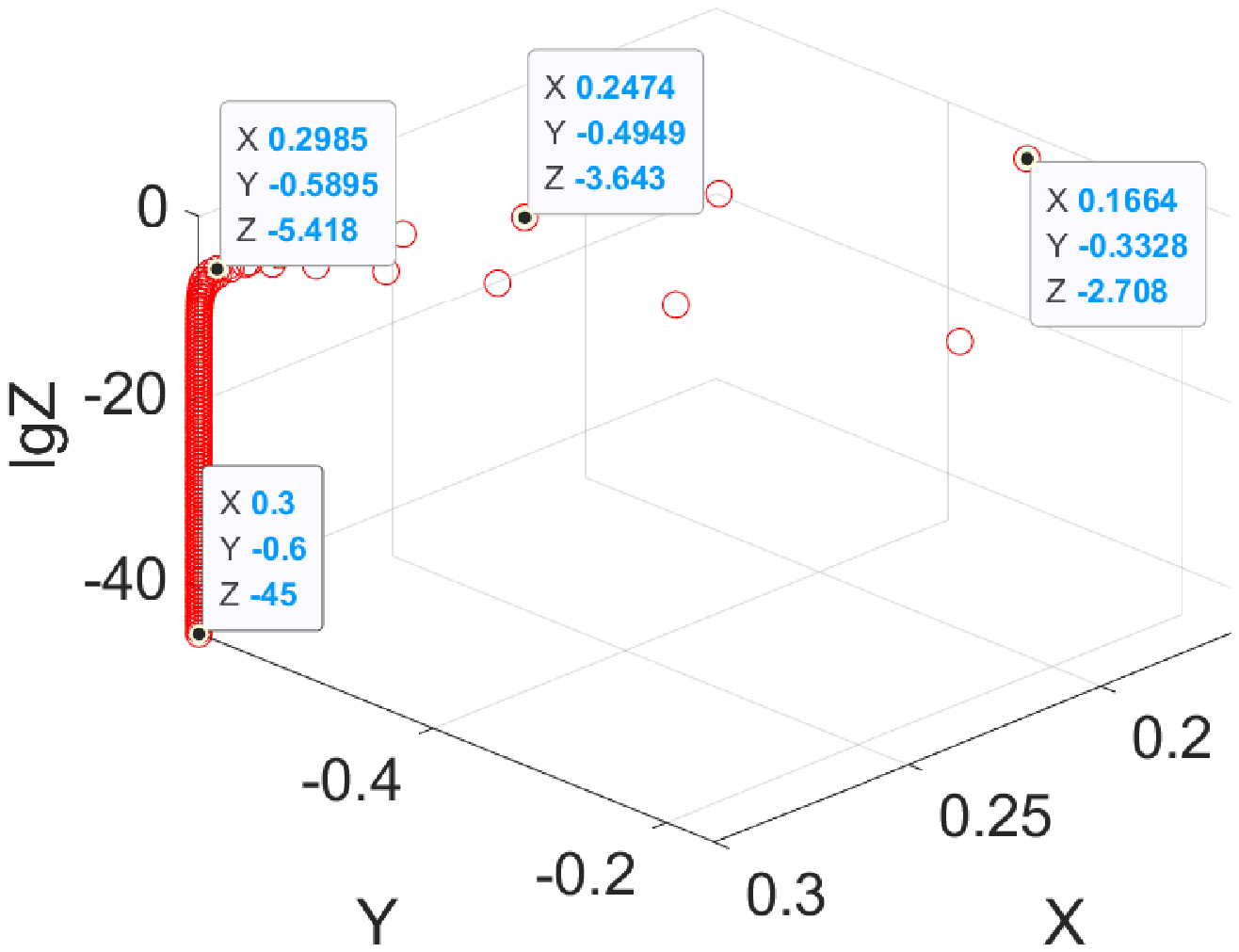}
  \includegraphics[width=5cm]{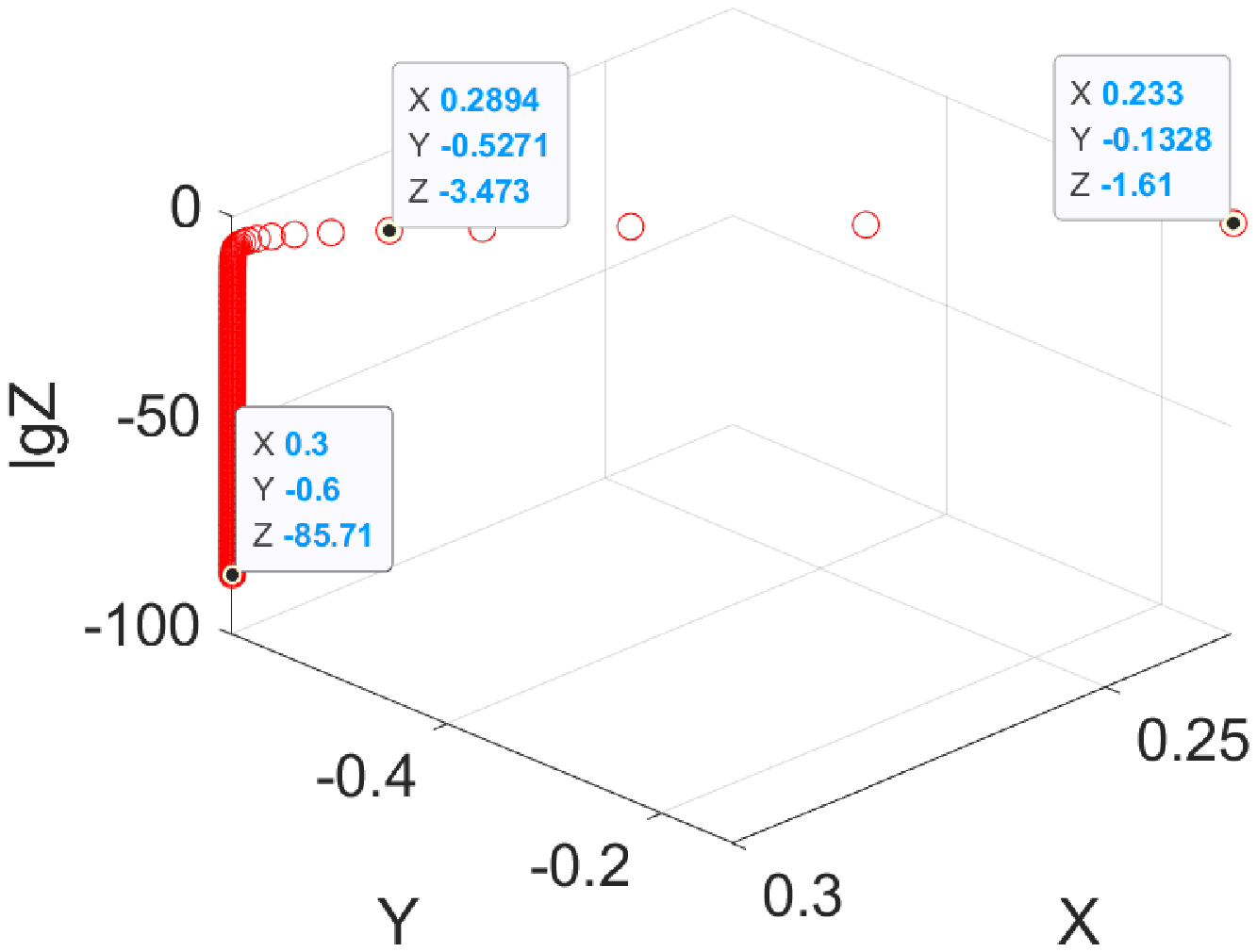}
     \caption{Behaviors of  Algorithm 3.2 for $\theta_n=0.5-1/(n+1)^5,\theta_n=0,\theta_n=1.$}
\label{Fig:7}
\end{figure}
Indeed,  if $\theta_n\equiv1$, then  Algorithm 3.1 coincides with the result of Moudafi\cite{M2018}. In addition,  from Figs.\ref{Fig:6}-\ref{Fig:7}, we can see that when $\theta_n=0$, two families of points alternate, and finally infinitely close to the exact null point $z=(3/10,-6/10,0)$.
\vskip 2mm

{\bf Example 5.2.} In this example, we show the behaviors of  our algorithms in $H=L^2([0,1])$. In addition, the compared results with  Dadashi and Postolach\cite{DP2020} are considered. In the simulation, we define  two mappings  $A$ and $B$ by $Bx(t)= \frac{x(t)}{2}$ and $Ax(t)=\frac{3x(t)}{4}$ for all $x(t)\in L^2([0,1])$. Then it can be shown that  $B$ is  $\frac{1}{2}-$ cocoercive.
 In the numerical experiment, the parameters are chosen as $\alpha_n=0.5-1/(10n+2)$, $\beta_n=\frac{1}{n+1}$ in Algorithm 3.2 for all $n\geq 1$.
 In addition, $\|x_{n+1}-x_n\|< 10^{-4}$ is used as stopping criterion and the following three different choices of initial functions $x_0(t), x_1(t)$ are chosen:
 \vskip 1mm
{\bf Case 1:}\,  $x_0(t)=\frac{\sin(-3t)+\cos(-5t)}{2}$ and $x_1(t)=2sin(5t) $;
\vskip 2mm
{\bf  Case 2:}\, $x_0(t)=\frac{2t\sin(3t)e^{-5t}}{200}$ and $x_1(t)=t^2-e^{-2t}$;
\newpage
{\bf Case 3:}\,  $x_0(t)=2t^3e^{5t}$ and $x_1(t)=\frac{e^{t}\sin(3t)}{100}$.

\vskip 2mm

Figs.\ref{Fig:1}-\ref{Fig:3} represent the numerical results for $\theta_n$ neither 0 nor 1, Fig.\ref{Fig:4} shows  the numerical results for $\theta_n=0$ and $\theta_n=1$, respectively. Fig.\ref{Fig:8} shows the comparising results with Dadashi and Postolach\cite{DP2020} for the initial points $x_0(t)=\frac{e^{t}\sin(3t)}{100}$  and $x_0(t)=t^2-e^{-2t}$, respectively. Table{\ref{Tab:1}} shows the comparisons  with Dadashi and Postolach\cite{DP2020} for the initial point $x_0(t)=t^2-e^{-2t}$.

%\begin{mdframed}
\begin{figure}[]
  \centering
   %Requires \usepackage{graphicx}
      \includegraphics[width=5.5cm]{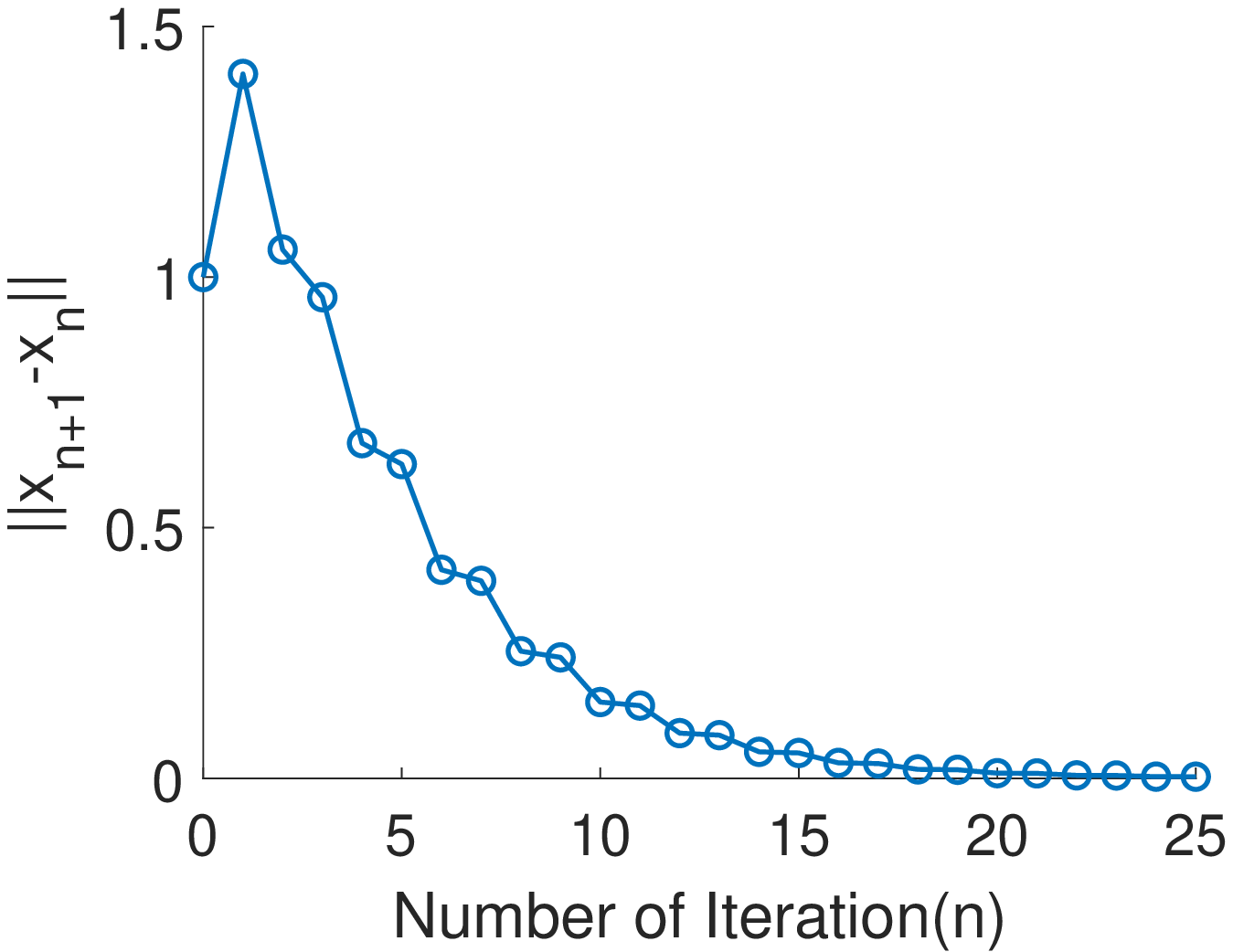}
   \includegraphics[width=5.5cm]{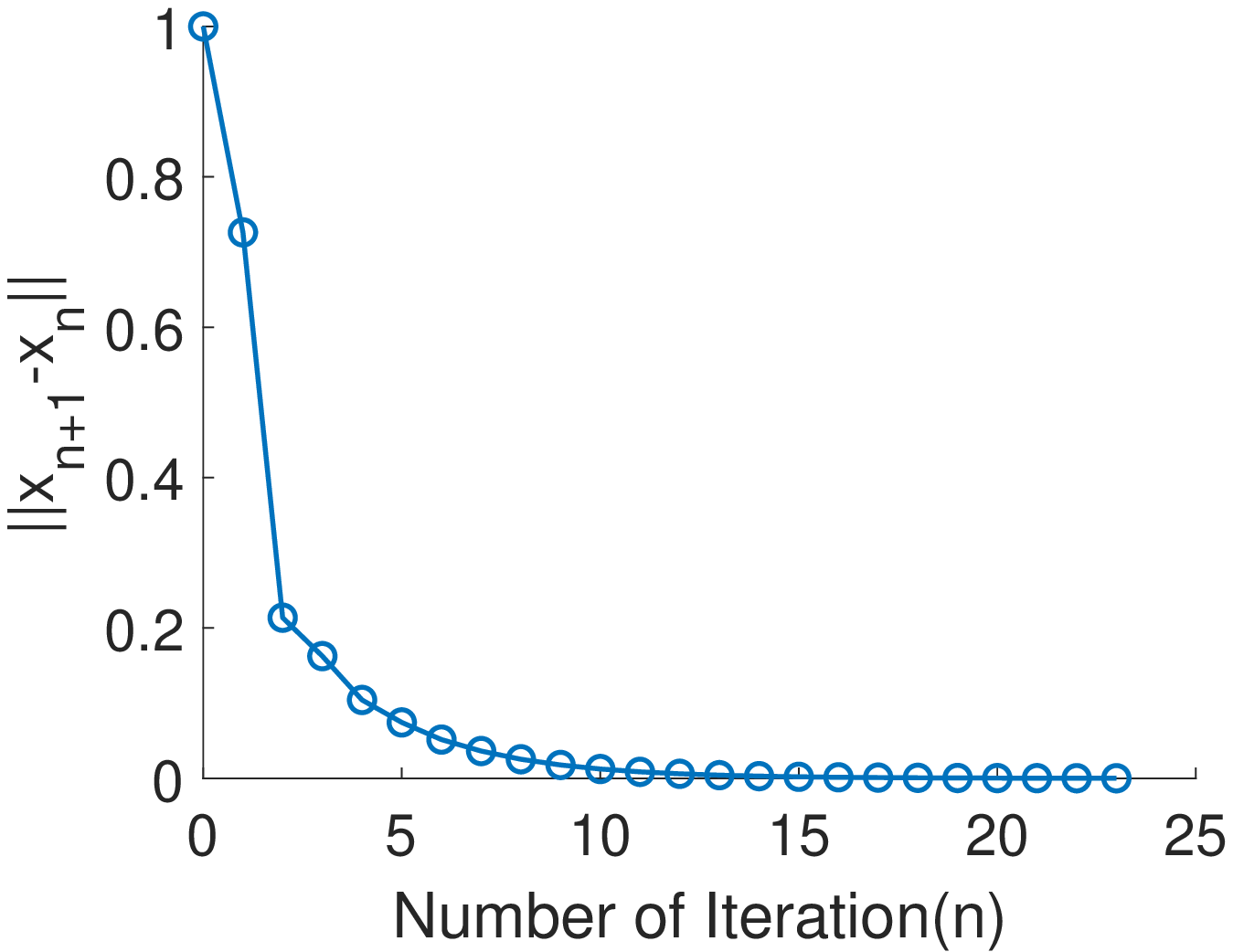}
   \caption{\small{Behavior of Algorithm 3.1 and Algorithm 3.2 for Case 1 in Exp.5.2.}}
  \label{Fig:1}
  %\end{mdframed}
  \end{figure}

\begin{figure}[]
  \centering
   %Requires \usepackage{graphicx}
   \includegraphics[width=5.5cm]{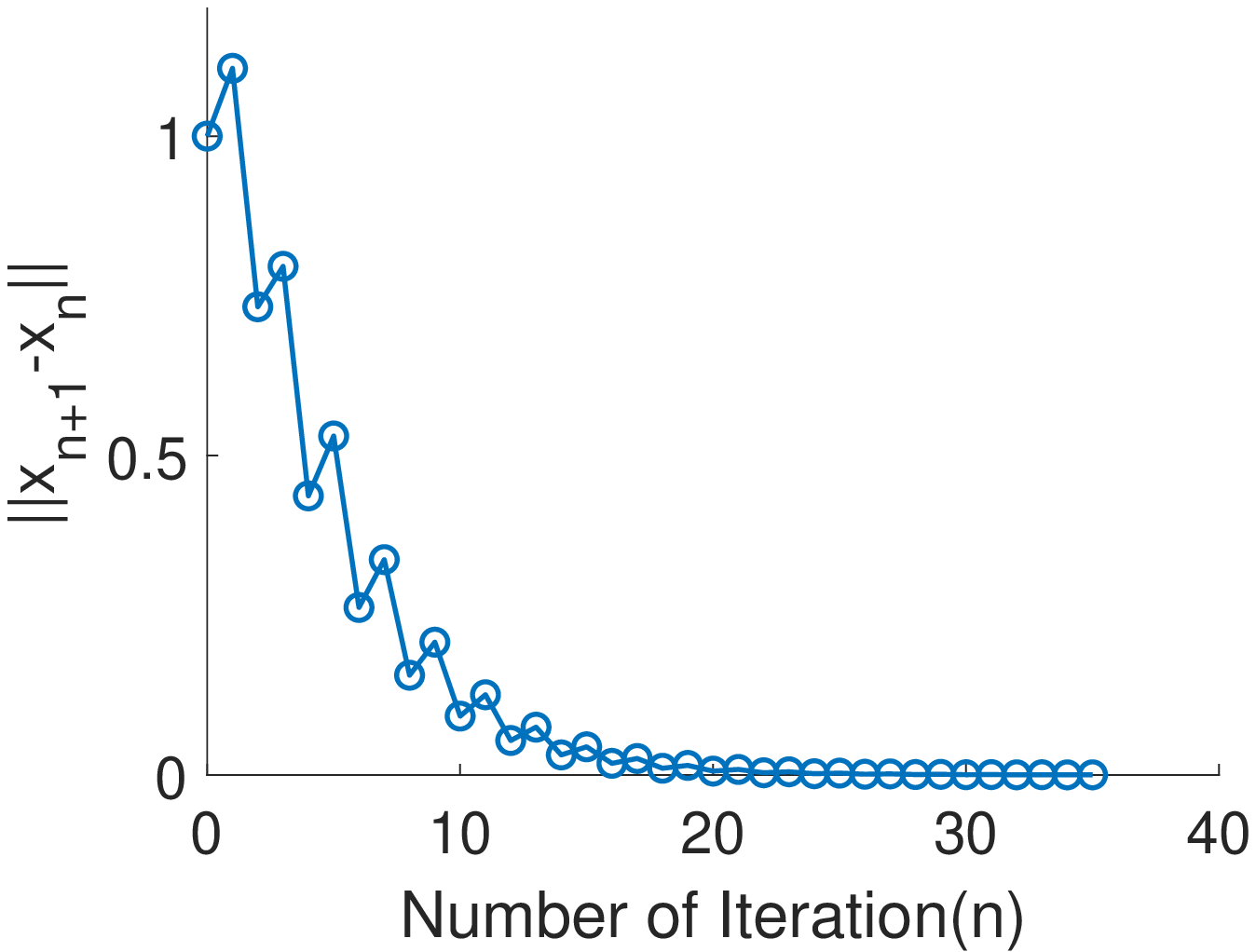}
   \includegraphics[width=5.5cm]{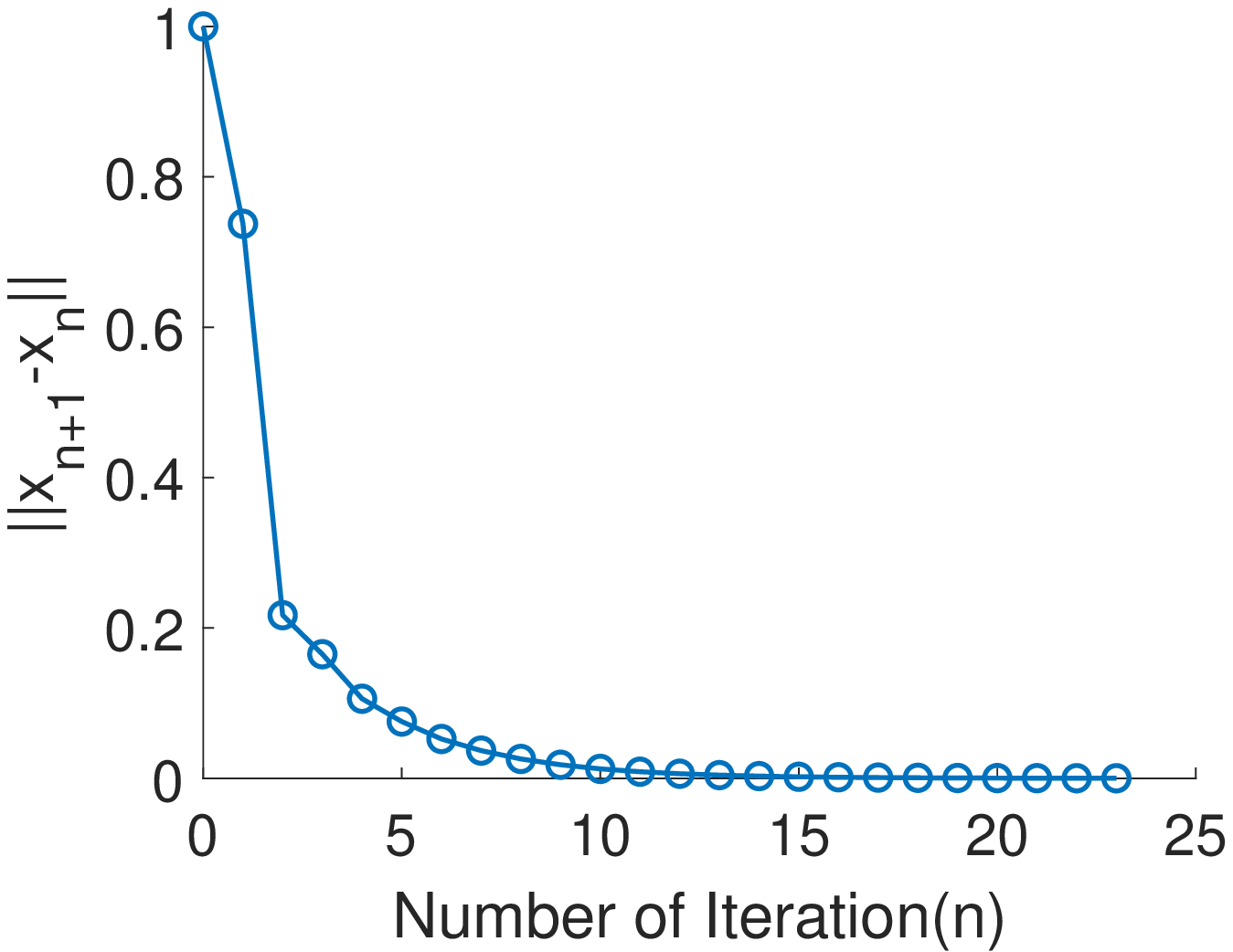}
   \caption{\small{Behavior of Algorithm 3.1 and Algorithm 3.2 for Case 2 in Exp.5.2.}}
  \label{Fig:2}
  \end{figure}
\begin{figure}[]
  \centering
   \includegraphics[width=5.5cm]{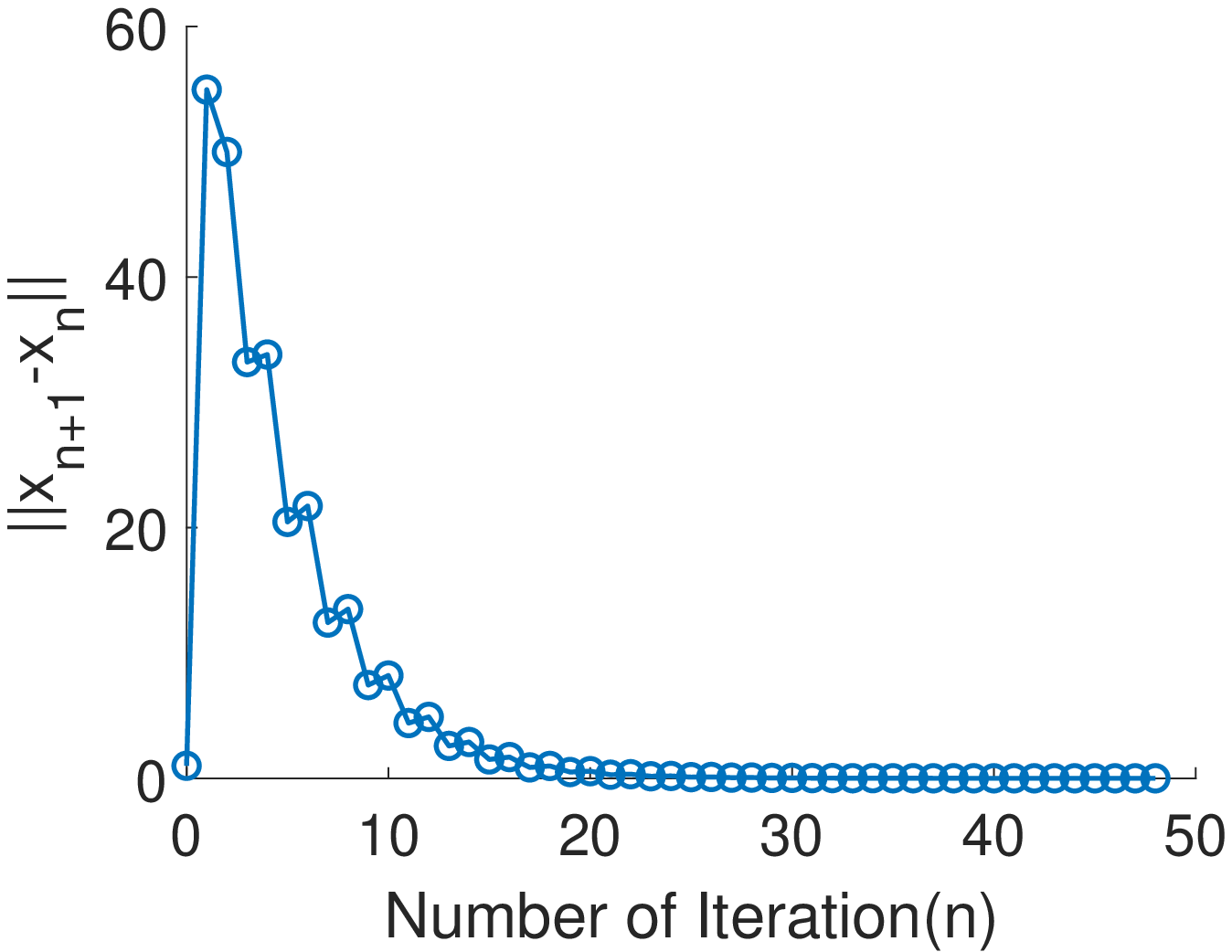}
   \includegraphics[width=5.5cm]{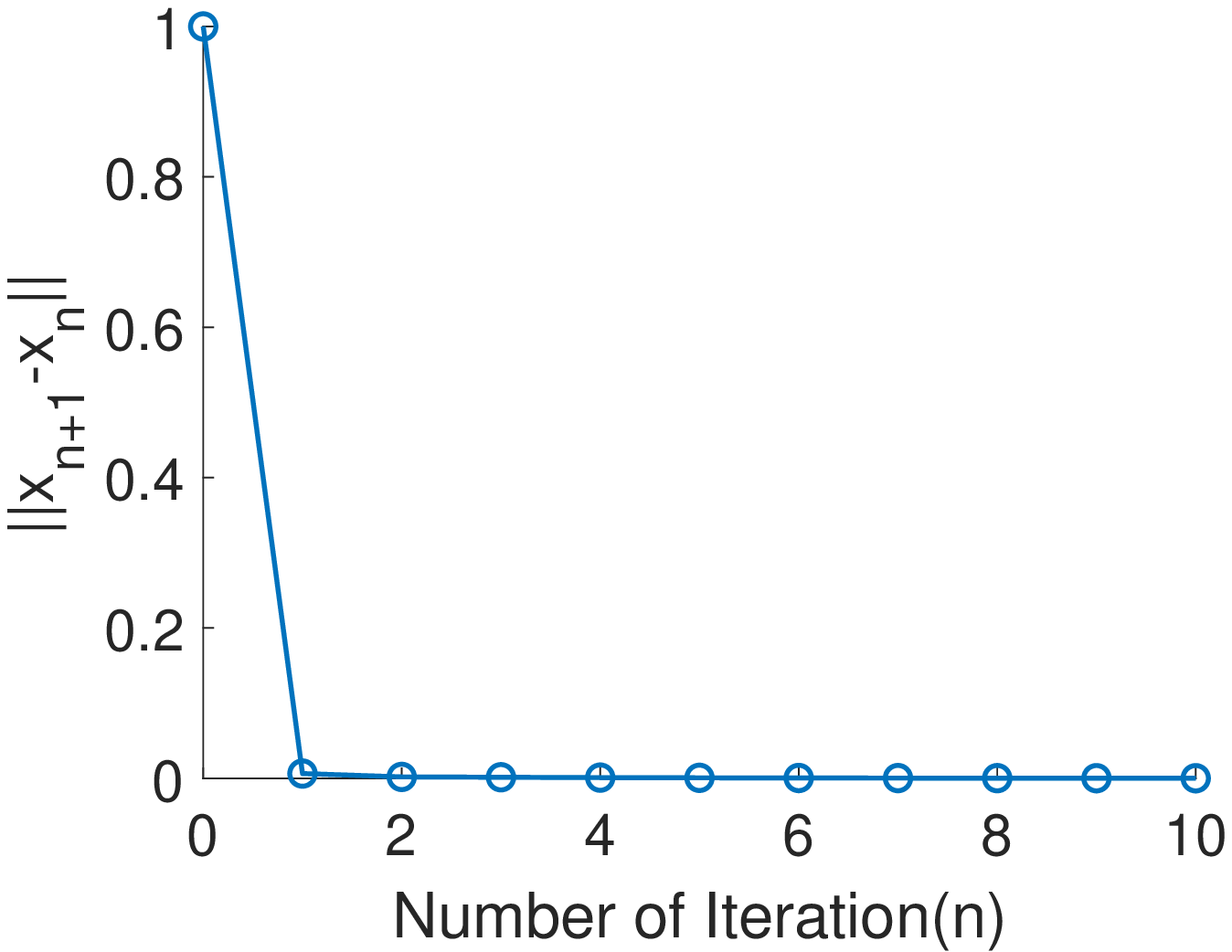}
   \caption{\small{Behavior of Algorithm 3.1 and Algorithm 3.2 for Case 3 in Exp.5.2.}}
  \label{Fig:3}
  \end{figure}
\begin{figure}[]
  \centering
   %Requires \usepackage{graphicx}
   \includegraphics[width=5.5cm]{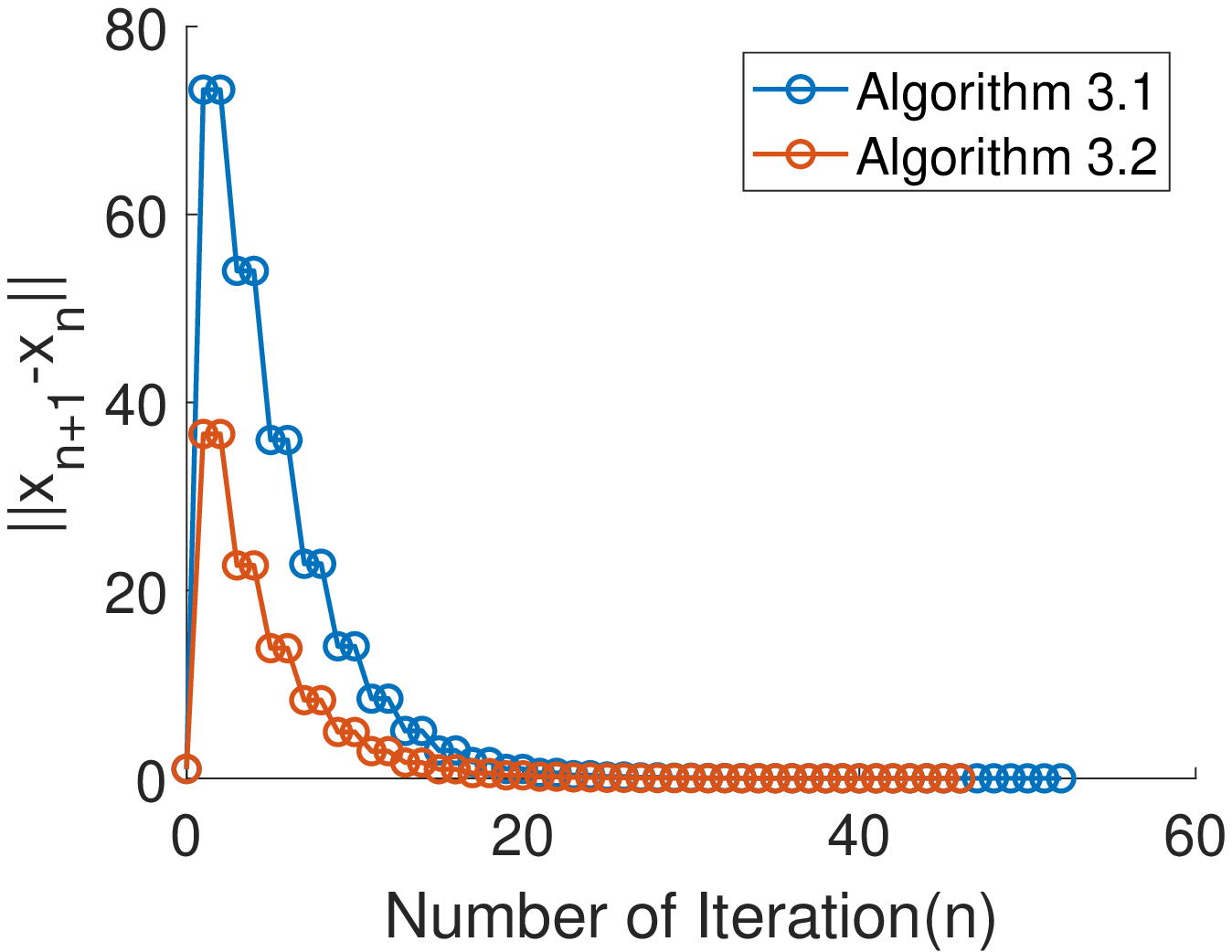}
   \includegraphics[width=5.5cm]{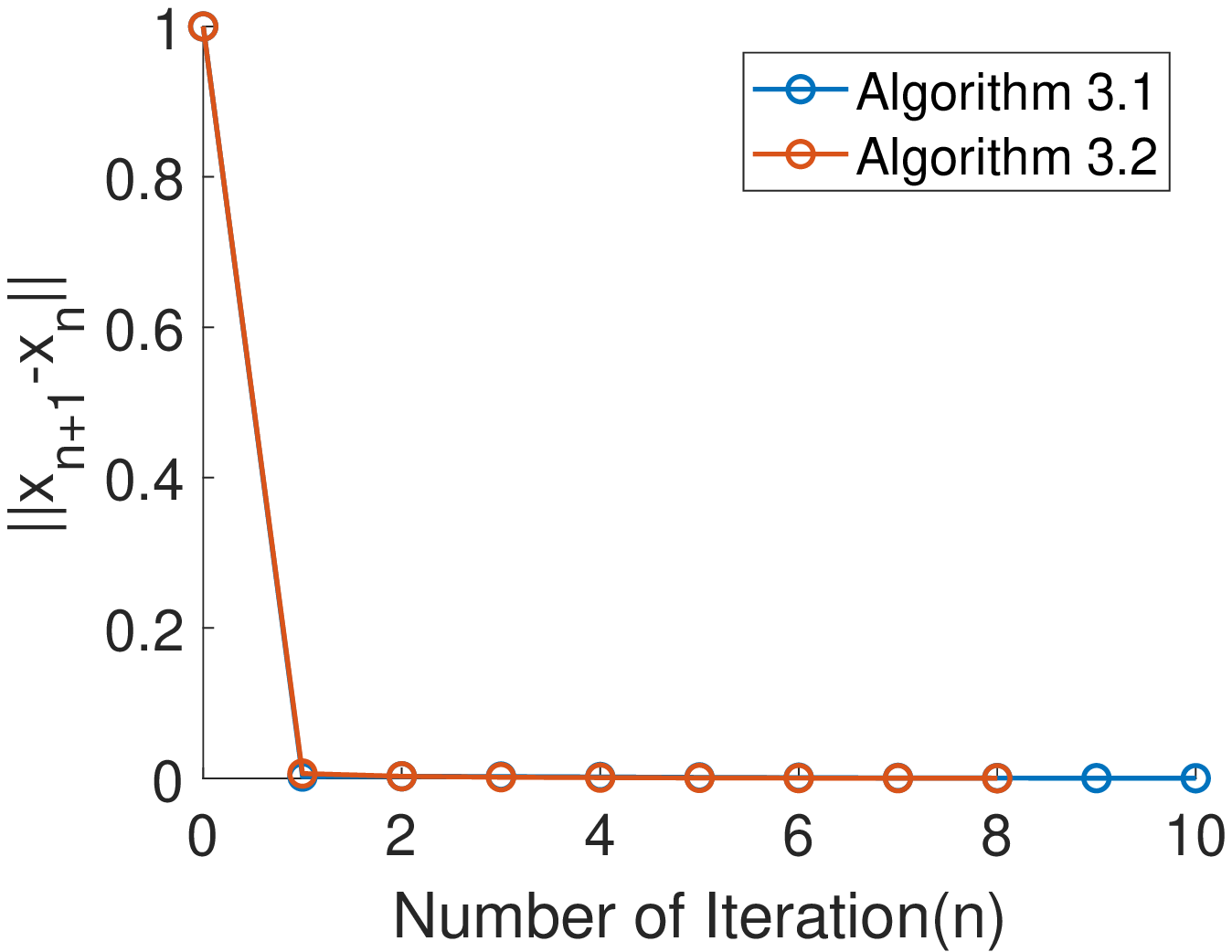}
   \caption{\small{Behavior of Algorithms 3.1 and 3.2 for $\theta_n=0$ and $\theta_n=1$ in Exp.5.2.}}
  \label{Fig:4}
  \end{figure}
\newpage
      \begin{figure}[]
  \centering
   %Requires \usepackage{graphicx}
   \includegraphics[width=6cm]{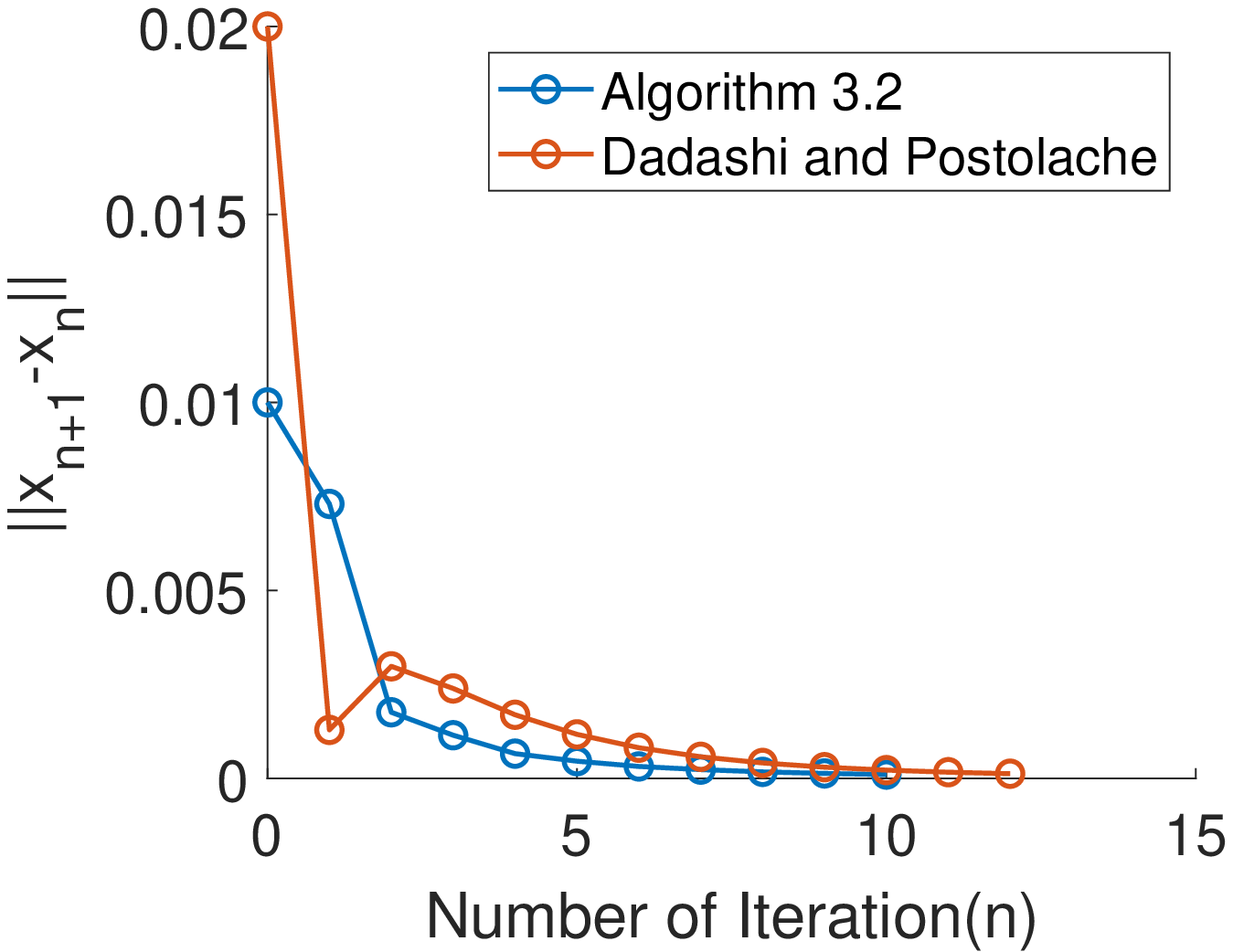}
   \includegraphics[width=6cm]{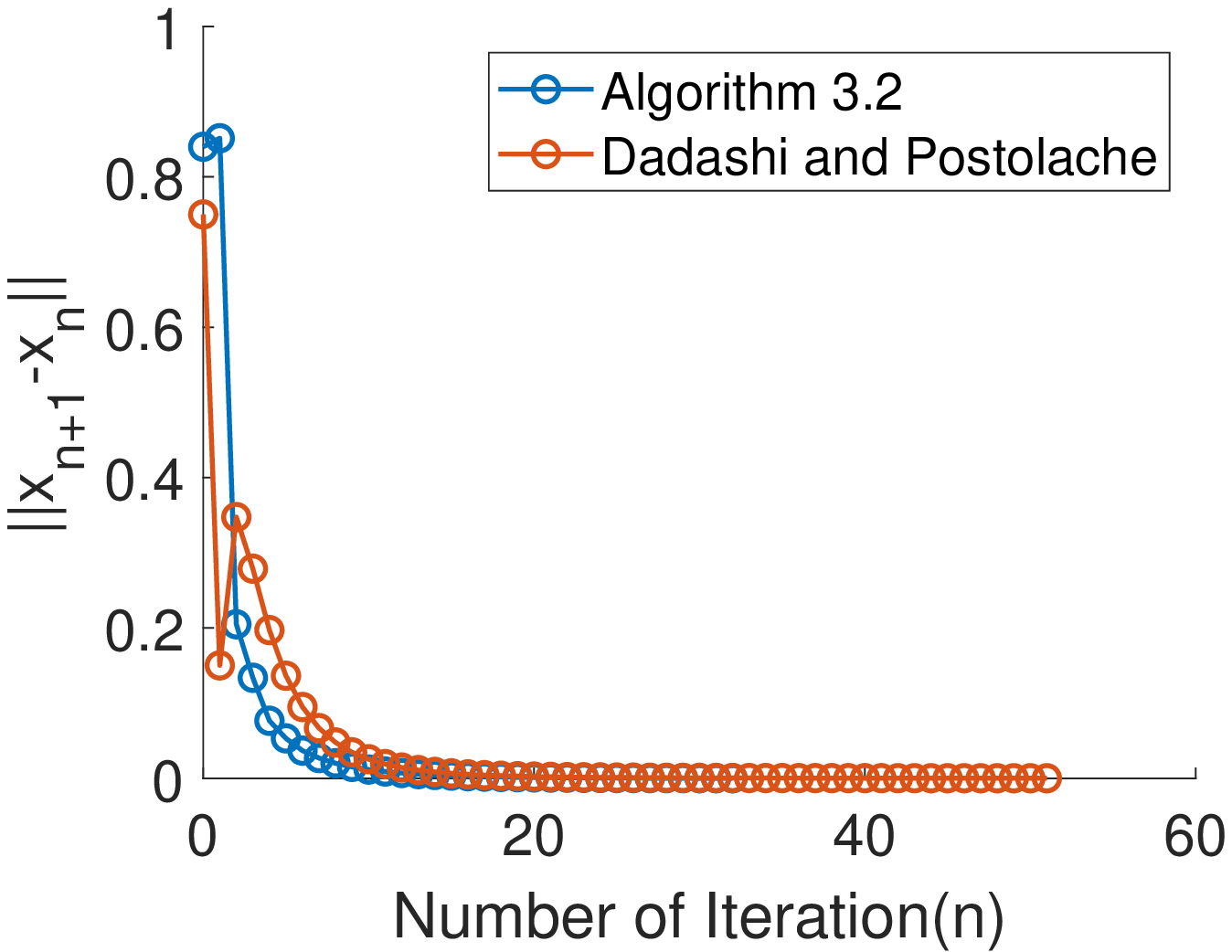}
   \caption{\small{Comparisons of Algorithm 3.2 with Dadashi and Postolach\cite{DP2020} in Exp.5.2.}}
  \label{Fig:8}
  \end{figure}

\begin{table}[]\small
  \centering
  \renewcommand{\arraystretch}{1}
  \caption{Comparisons of Algorithm 3.1, Algorithm 3.2 and Dadashi and Postolach\cite{DP2020}}
\label{Tab:1}
\begin{tabular}{  c c c c c c c }
  \hline
\hline
Error&Iter.&   & &CPU in second & & \\
\cline{2-4}
\cline{5-7}
 &Alg.1&Alg.2&Alg. in Dadashi et al.[31]&Alg.1&Alg.2 &Alg. in Dadashi et al.[31] \\
\hline
$10^{-3}$&41&21    &26         &6.2813&3.4688 &3.8438 \\
$10^{-4}$&49&33  &52     &10.0445&7.6542 &7.8461 \\
$10^{-5}$&57&45   &110   &10.2969&8.6768 &24.3212 \\
\hline
\end{tabular}
\end{table}
It is clear that our algorithms, especially Algorithm 3.2, is  faster, more efficient,
more stable.
\vskip 2mm
{\bf Example 5.3.} (Compressed Sensing) In this example, to show the effectiveness and applicability of our algorithms  in the real world, we consider the recovery of a sparse and noisy signal from a limited number of sampling which is a problem from the field of compressed sensing.  The sampling matrix $T\in R ^{m*n}$, $m<n$, is stimulated by standard Gaussian distribution and vector $b=Tx+\epsilon$, where $\epsilon$ is additive noise. When $\epsilon=0$, it means that there is no noise to the observed data. For further explanations the reader can consult Nguyen and Shin\cite{NS2013}.
\vskip 2mm
 Let $x_0\in R^n$ be the $K$-sparse signal, where $K<<n$. Our task is to recover the signal $x_0$ from the data $b$. To this end, we  transform it into finding the solution of LASSO problem:
\begin{eqnarray*}
&\min_{x\in R^n}\|Tx-b\|^2\\
&s.t.\,\, \|x\|_1\leq t,
\end{eqnarray*}
where $t$ is a given positive constant. If we define
\begin{equation*}
B_1(x)=
\begin{cases}
 \{u:\sup_{\|x\|_1\leq t}\langle x-y,u\rangle\leq 0\}, &\hbox{if}\,\, y \in R^n,\\
 \emptyset, &\hbox{otherwise},
\end{cases}
\quad
B_2(x)=
\begin{cases}
 R^m,&\hbox{if}\,\, x=b,\\
\emptyset,&\hbox{otherwise},
\end{cases}
\end{equation*}
then one can see that the LASSO problem coincides with the  problem of finding $x^*\in R^n$ such that
$$
0\in B_1(x^*)\,\,\, \hbox{and}\,\,\, 0\in B_2(Tx^*),
$$
 which associates with the the problem for finding the solution of the problem:
 $$
 0 \in (A+B)x^*,
 $$
 where  $A=B_1$ and  $B=T^*(I-J_\lambda B_2)T$ with  $\gamma-$ cocoercive, $\gamma=\frac{1}{\|T^*T\|}$.
\vskip 2mm
In addition to showing the behavior of our algorithms, the results of Sitthithakerngkiet et al. \cite{SDMK2018},  Kazimi and Riviz \cite{KR2014} without inertial process and Tang\cite{Tang2019} with general inertial method are compared. For the experiment setting, we choose the following parameters: $T\in R^{m*n}$ is generated randomly with $m=2^6,2^7$, $n=2^8,2^9$, $x_0\in R^n$ is $K$-spikes $(K=40, 60)$ with amplitude $\pm{1}$ distributed in whole domain randomly.
\vskip 2mm
 For simplicity, we define the nonexpansive mappings $S_n:\mathbb{R}^3\rightarrow\mathbb{R}^3$ as
  $S_n=I$  for Algorithm 3.1 in Sitthithakerngkiet et al. \cite{SDMK2018}, and the strongly positive
  bounded linear operator $D=I$, the constant $\xi=0.5$
  and we give fixed point $u=0.1$. Moreover, we take $\alpha_n=0.5-1/(10*n+2)$, $\beta_n=10^{-3}/(n+1)$ in all compared algorithms.
  Let $t=K-0.001$  and $\|x_{n+1}-x_n\|\leq 10^{-4}$ be the stopping criterion. All the numerical results are presented in Table \ref {Tab:2} and Fig.\ref{Fig:5}.

\begin{table}[]\small
  \centering
  \renewcommand{\arraystretch}{1}
  \caption{The comparisons of Algorithm 3.1, Algorithm 3.2,
  Sitthithakerngkiet et al.  \cite{SDMK2018}, Kazmi et al. \cite{KR2014}, Tang\cite{Tang2019}}
\label{Tab:2}
\begin{tabular}{  c c c c c }
  \hline
  Method&$K=40$,$m=2^6$,$n=2^8$ & & $K=60$,$m=2^7$,$n=2^9$\\
   \hline
   & Iter.\hspace{1cm}Sec.& &Iter.\hspace{1cm}Sec.\\
   \hline
 Algorithm 3.1& 43\hspace{1cm} 0.0992& & 83 \hspace{1cm} 0.4431 \\
Algorithm 3.2&63 \hspace{1cm}0.1635 & & 79 \hspace{1cm} 0.6368     \\
  Sitthithakerngkiet et al. \cite{SDMK2018}&91\hspace{1cm} 0.12 & & 122\hspace{1cm} 0.6942\\
 Kazmi et al. \cite{KR2014}&54\hspace{1cm}0.0771& &39\hspace{1cm}0.1981\\
   Algo.3.2-Tang\cite{Tang2019} & 54 \hspace{1cm} 2.2632 &&67\hspace{1cm}3.2541\\
   \hline
\end{tabular}
\end{table}

\begin{figure}[]
  \centering
  \includegraphics[width=8cm]{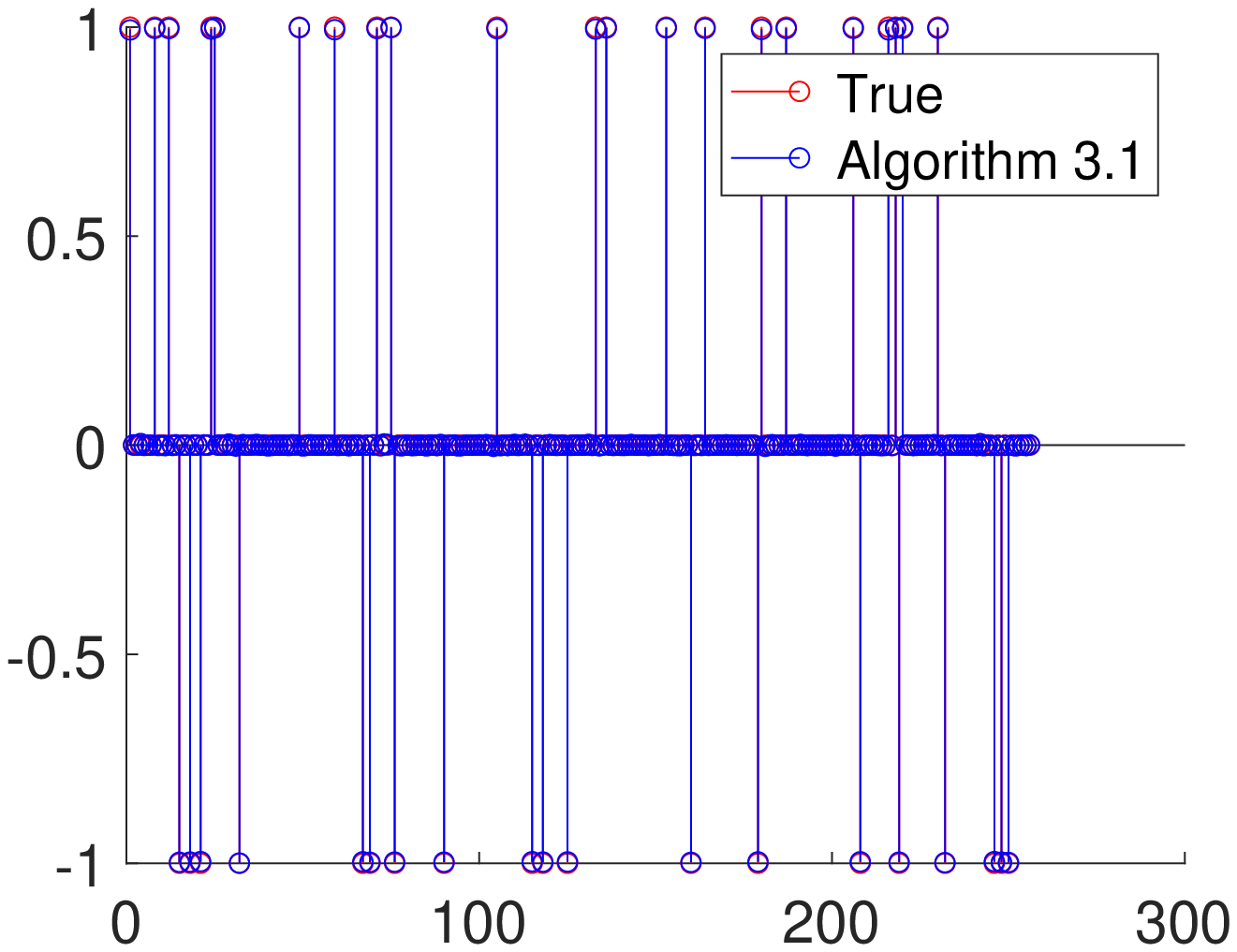}
  \includegraphics[width=8cm]{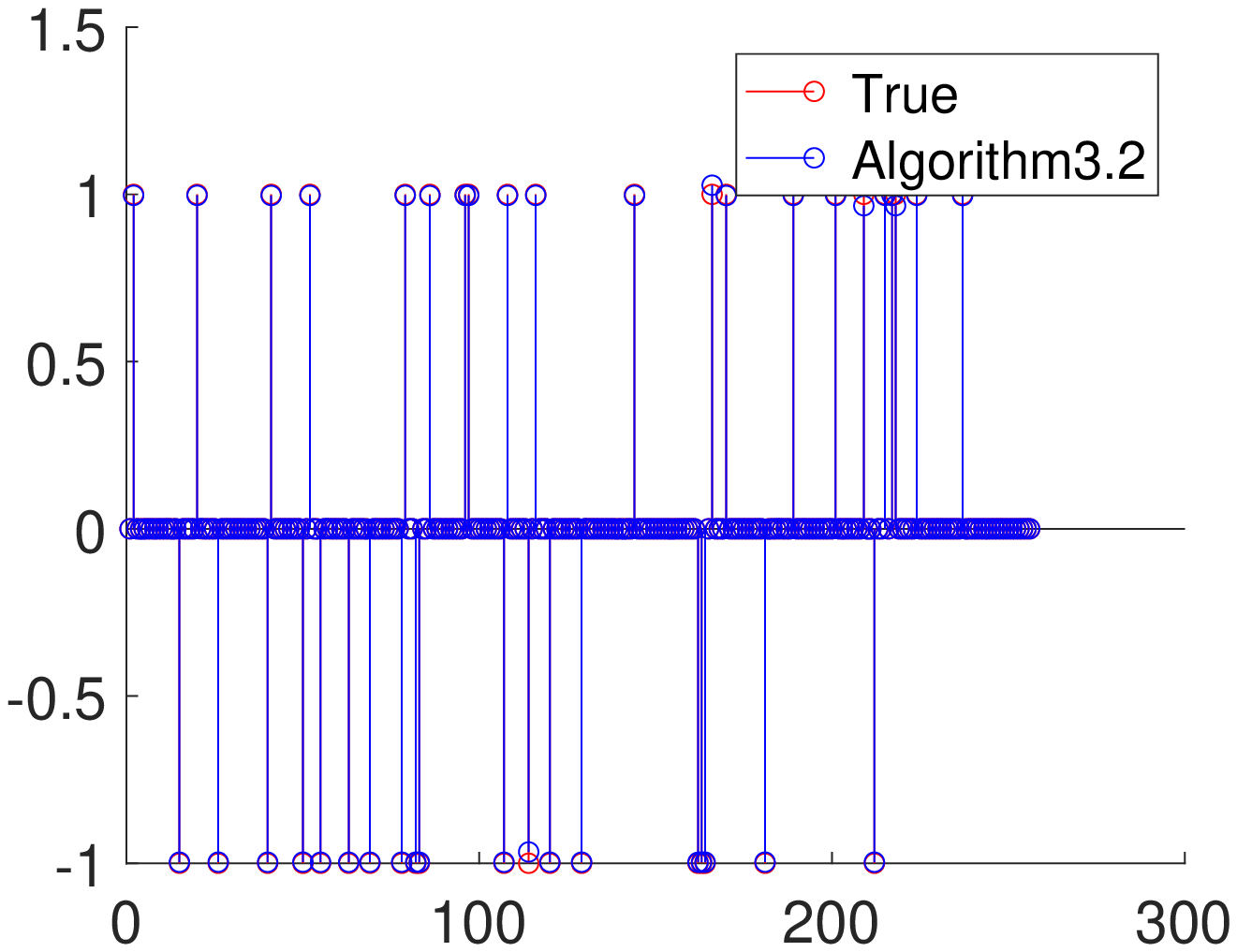}
  \includegraphics[width=8cm]{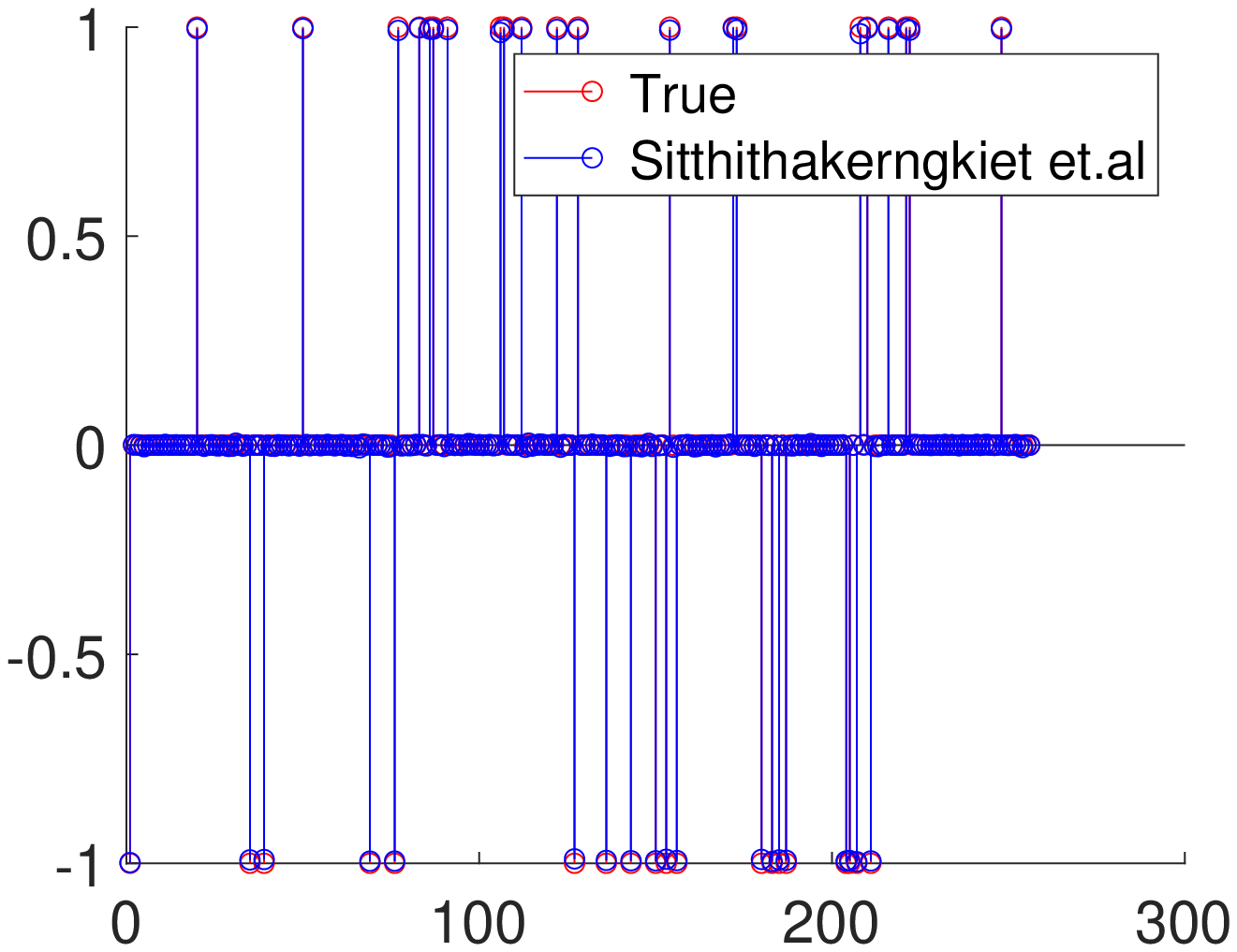}
  \includegraphics[width=8cm]{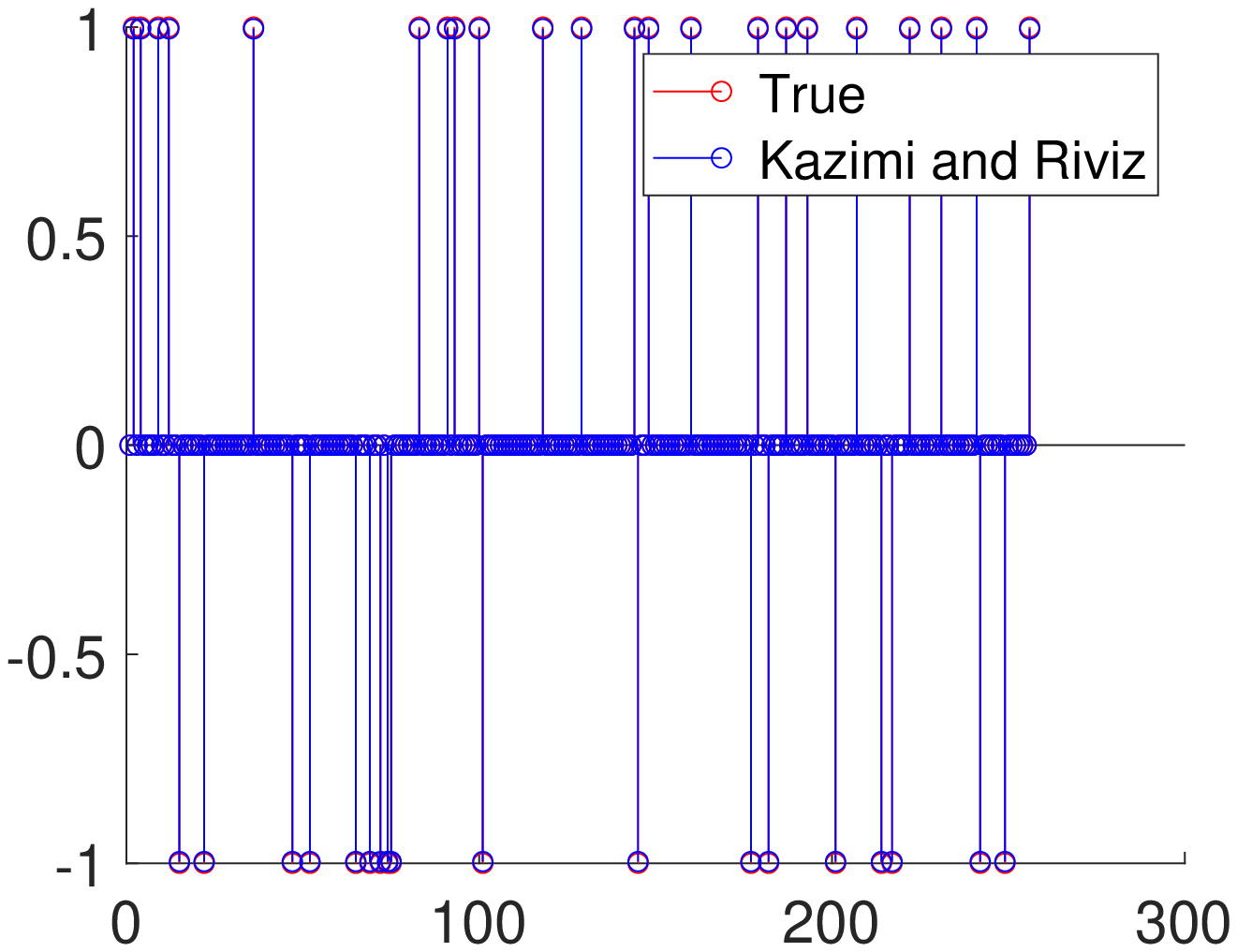}
  \includegraphics[width=8cm]{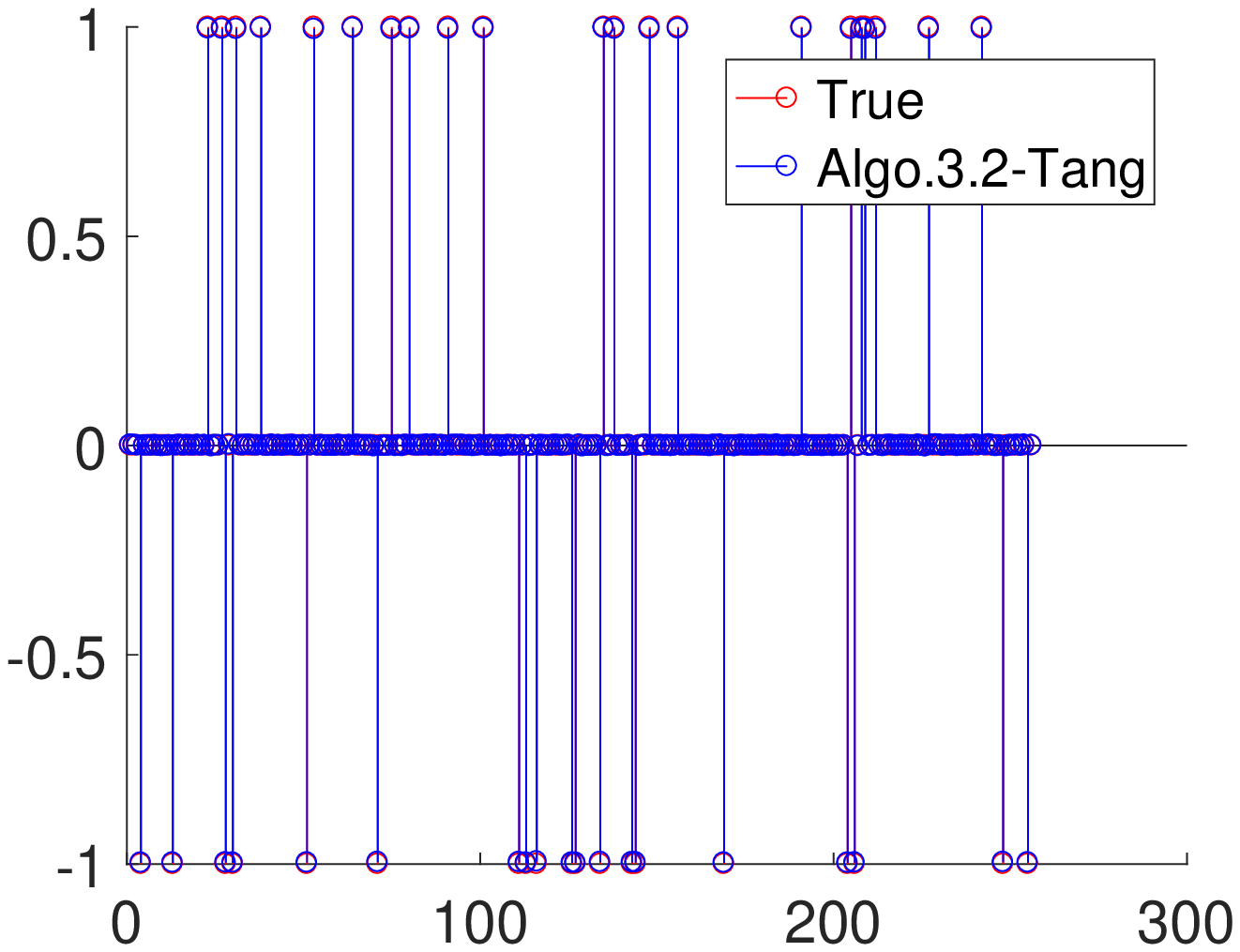}
  \includegraphics[width=8cm]{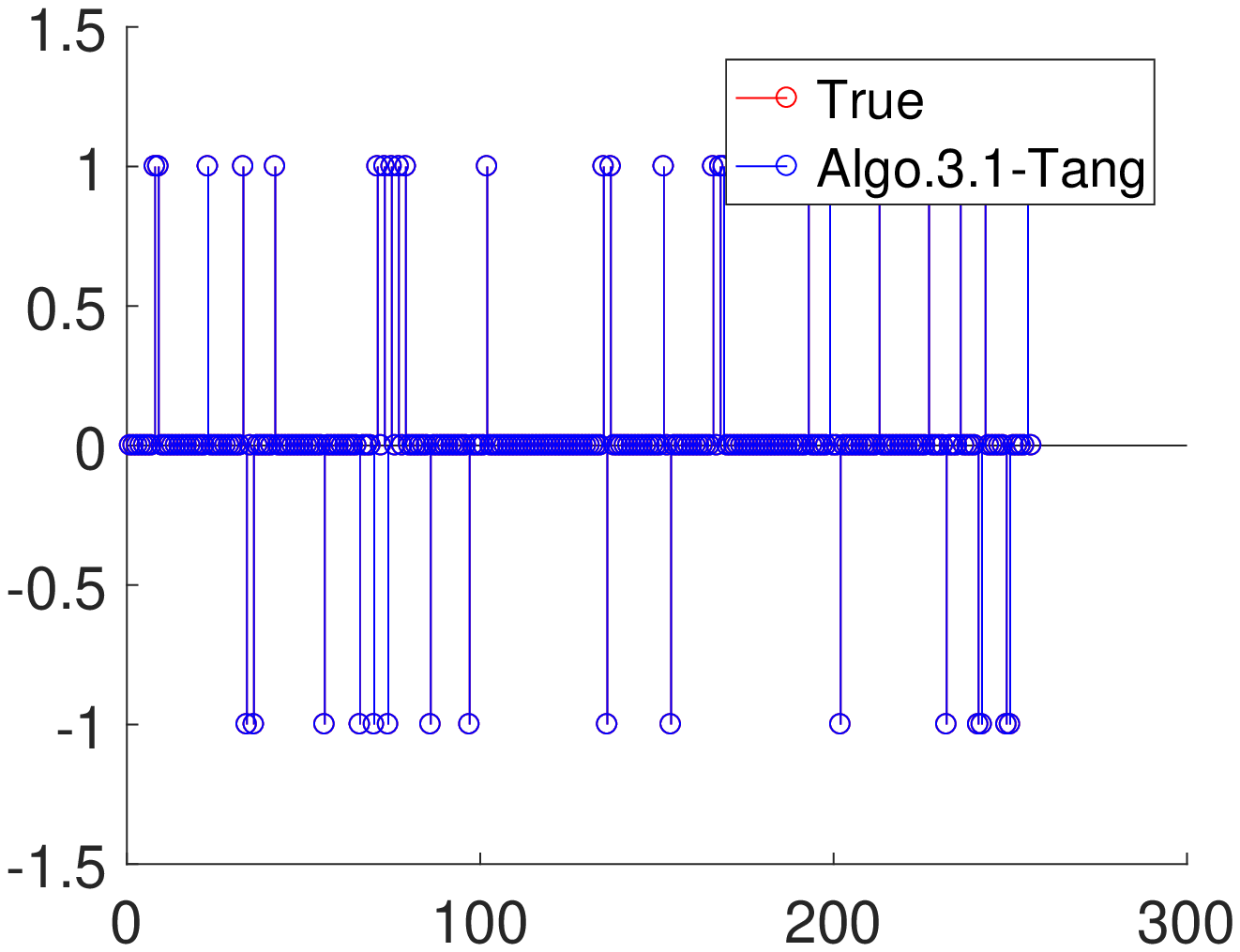}
   \caption{Numerical results for $m=2^6$, $m=2^8$ and $K=40$}
\label{Fig:5}
\end{figure}
\newpage

%\begin{figure}
  %\centering
    %\includegraphics[width=7cm]{Fig10.eps}
      % \includegraphics[width=7cm]{Fig11.eps}
    %\includegraphics[width=7cm]{Fig12.eps}
    %\includegraphics[width=7cm]{Fig13.eps}
   % \includegraphics[width=7cm]{19.eps}
    %\includegraphics[width=7cm]{Fig14.eps}
    % \includegraphics[width=7cm]{23.eps}
     %  \caption{\small{Numerical results for $m=2^6,m=2^8$ and $K=30$.}}
   % \label{Fig:4}
%\end{figure}

%\begin{figure}
 % \centering
   % \includegraphics[width=7cm]{15.eps}
      % \includegraphics[width=7cm]{16.eps}
   %\includegraphics[width=7cm]{Fig12.eps}
    %\includegraphics[width=7cm]{17.eps}
    %\includegraphics[width=7cm]{18.eps}
   % \includegraphics[width=7cm]{20.eps}
     %\includegraphics[width=7cm]{21.eps}
      %\includegraphics[width=7cm]{22.eps}
      % \caption{\small{Numerical results for $m=2^7,m=2^9$ and $K=50$.}}
   % \label{Fig:4}
%\end{figure}

\newpage
\section{Conclusion}

We have provided two new {\it Inertial-Like Proximal} iterative algorithms (Algorithms 3.1 and 3.2) for the  {\it null point problem} .
We have proved that, under some mild conditions,  Algorithms 3.1
converges weakly and Algorithms 3.2 strongly to a solution of the {\it null point problem}. Thanks to the novel structure of the inertial-like technique, our Algorithms 3.1 and 3.2 seem to have the following merits:
\vskip 1mm
(i) Theoretically, they do not need verify the traditional and difficult checking condition $\sum_{n=1}^\infty\theta_n\|x_n-x_{n-1}\|^2<\infty$,  namely,  our convergence theorems still hold even without this condition.
\vskip 1mm

 (ii) Practically, they do not involve the computations of the norm of the difference between $x_n$ and $x_{n-1}$ before choosing the inertial parameter $\theta_n$(hence less computation cost), as opposed to almost all previous inertial algorithms in the existing literature, that is, the constraints on inertial parameter $\theta_n$ are looser and more natural,  so they are extremely attractive and friendly for users.

\vskip 1mm
(iii) Different from the  general inertial algorithms, the inertial factors $\theta_n$ in our {\it Inertial-Like Proximal} algorithms are chosen in $[0,1]$ with $\theta_n=1$ possible, which are new, natural  and interesting algorithms in some ways. In particular,  under more mild assumptions, our proofs are simpler and different from the others.
\vskip 1mm
We have included several numerical examples which show the efficiency and reliability of Algorithm 3.1 and Algorithm 3.2.
We have also made comparisons of Algorithm 3.1 and Algorithm 3.2 with other four algorithms in Sitthithakerngkiet et al. \cite{SDMK2018},  Kazimi and Riviz \cite{KR2014}, Tang\cite{Tang2019}, Dadashi and Postolach\cite{DP2020}
confirming some advantages of our novel inertial algorithms.\\

%{\bf Competing Interests} The authors declare that they have no competing interests.
%\vskip 2mm

%\noindent

% {\bf Authors' Contributions} All authors contributed equally to
% this work. All authors read and approved final manuscript.
%\vskip 2mm

\noindent {\bf Acknowledgements.}
  This article was funded by the Science Foundation of China (12071316),
   Natural Science Foundation of Chongqing
(CSTC2019JCYJ-msxmX0661), Science and Technology Research Project
of Chongqing Municipal Education Commission (KJQN 201900804) and
the Research Project of Chongqing Technology
and Business University (KFJJ1952007).\\

\end{document}